\documentclass[oneside]{amsart}
\usepackage[utf8]{inputenc}
\usepackage[english]{babel}
\usepackage{amssymb,amsmath,amsthm}
\usepackage{graphicx}
\usepackage[small,nohug,heads=vee]{diagrams}
\diagramstyle[labelstyle=\scriptstyle]
\usepackage{xspace}
\usepackage{enumerate}
\usepackage{verbatim}
\usepackage{geometry} 

\newcommand{\QQ}{\ensuremath{\mathbb{Q}}\xspace}

\newcommand{\ZZ}{\ensuremath{\mathbb{Z}}\xspace}
\newcommand{\Zp}{\ensuremath{\mathbb{Z}_{(p)}}\xspace}

\newcommand{\NN}{\ensuremath{\mathbb{N}}\xspace}
\newcommand{\LL}{\ensuremath{\mathbb{L}}\xspace}

\newcommand{\rarr}{\rightarrow}

\newcommand{\xrarr}[1]{\xrightarrow{#1}}

\newcommand{\mf}[1]{\ensuremath{\mathfrak{#1}}}

\newcommand{\id}{\ensuremath{\mbox{id}}}

\newcommand{\ot}{\otimes}

\newcommand{\Ext}{\ensuremath{\mathrm{Ext}}}
\newcommand{\Tor}{\ensuremath{\mathrm{Tor}}}

\makeatletter
\newcommand{\colim@}[2]{%
  \vtop{\m@th\ialign{##\cr
    \hfil$#1\operator@font colim$\hfil\cr
    \noalign{\nointerlineskip\kern1.5\ex@}#2\cr
    \noalign{\nointerlineskip\kern-\ex@}\cr}}%
}
\newcommand{\colim}{%
  \mathop{\mathpalette\colim@{\rightarrowfill@\textstyle}}\nmlimits@
}
\makeatother

\DeclareMathOperator{\Spec}{\mathrm{Spec}}
\DeclareMathOperator{\Td}{\mathrm{Td}}

\DeclareMathOperator{\F}{\mathrm{F}}
\DeclareMathOperator{\CH}{\mathrm{CH}}
\DeclareMathOperator{\Ann}{\mathrm{Ann}}

\theoremstyle{definition}
\newtheorem{Def}{Definition}[section]
\newtheorem*{Def-intro}{Definition}
\newtheorem{Rk}[Def]{Remark}

\theoremstyle{plain}
\newtheorem{Th}[Def]{Theorem}
\newtheorem*{Th-intro}{Theorem}
\newtheorem{Prop}[Def]{Proposition}
\newtheorem{Cr}[Def]{Corollary}
\newtheorem{Lm}[Def]{Lemma}
\newtheorem*{Conj}{Syzygies Conjecture for Algebraic Cobordism}

\newtheorem*{BigTh-add}{Algebraic Classification of Additive Operations Theorem  (CAOT)}

\begin{document}

\title{On the Structure of Algebraic Cobordism}
\author{Pavel Sechin}
\date{}
\begin{abstract}
In this paper we investigate the structure of algebraic cobordism of Levine-Morel
as a module over the Lazard ring with the action of Landweber-Novikov and symmetric operations on it.
We show that the associated graded groups of algebraic cobordism
with respect to the topological filtration $\Omega^*_{(r)}(X)$ 
are unions of finitely presented $\LL$-modules of very specific structure.
Namely, these submodules possess a filtration
 such that the corresponding factors are either free or isomorphic
 to cyclic modules $\LL/I(p,n)x$ where $\deg x\ge \frac{p^n-1}{p-1}$.
As a corollary we prove the Syzygies Conjecture of Vishik 
on the existence of certain free $\LL$-resolutions of $\Omega^*(X)$,
and show that algebraic cobordism of a smooth surface can be described
in terms of $K_0$ together with a topological filtration.
\end{abstract}

\maketitle

\section{Introduction}

Complex cobordism $MU^*$ is an example of a generalized cohomology theory 
which contains much more information
 about the topological manifold than ordinary (aka singular) cohomology.
For example, the Conner-Floyd theorem says that $K$-theory can be obtained as a factor of $MU^*$
over a certain ideal.
On the other hand, rational cobordism $MU^*\ot\QQ$ are
canonically isomorphic to singular cohomology with the coefficients
in an infinitely generated polynomial ring $\QQ[t_1, t_2\ldots]$,
and hence the 'new' information in $MU^*$ is contained in its integral structure.
The coefficient ring of $MU^*$ 
is a polynomial ring in a countable number of variables $\ZZ[x_1, x_2\ldots]$,
and, thus, $MU^*(X)$ is a module over this ring for every space $X$.
However, not every module over this ring can be realised in such a way,
and one of the areas of research is a study of restrictions on $MU^*(X)$ as $MU^*(pt)$-module.

One of the first insights into the structure of complex cobordism appeared
when Landweber \cite{Land-op} and Novikov \cite{Nov}
 independently have calculated all stable operations
in complex cobordism $MU^*$ (in fact, they worked with co-operations on complex bordism,
but the groups of operations and co-operations turn out to be dual to each other). 
The algebra of these operations is now called the Landweber-Novikov algebra,
and it is generated by the graded components of the so-called total Landweber-Novikov operation.

Landweber later has proved foundational results
on the structure of any $MU^*(pt)$-module on which Landweber-Novikov operations act 
(\cite{Land2, Land3, Land}).
Such algebraic objects can be identified with comodules over the Hopf algebroid 
$(MU_*, MU_*(MU))$\footnote{Topologists tend to denote $MU^n(pt)$ as $MU_{-n}=MU_{-n}(pt)$ 
and we follow this convention.}.
More precisely, any such comodule which is finitely presented as a module over $MU_*$
possesses a finite filtration by comodules, such that its graded factors
 are isomorphic to $MU_*/I(p_i,n_i)$
where $I(p_i,n_i)$ is a certain ideal in $MU_*$  corresponding to a prime number $p_i$
and a natural number $n_i$. In particular, this allowed Landweber to find a big class of $MU_*$-modules,
so-called {\it Landweber-exact} modules,
s.t. for any such module $K$ 
the group $\Tor^{MU_*}_i(MU^*(X),K)$ is zero for any CW-complex $X$ and any $i>0$,
and $MU^*\ot_{MU_*} K$ becomes a generalized cohomology theory.

Quillen was first to notice that the above structure on $MU^*$ 
comes as a corollary to the fact that the ring of complex cobordism over a point
is canonically isomorphic to the Lazard ring classifying formal group laws (\cite{Quil}).
In particular, this allowed to reinterpret existing at the moment theory of Brown-Peterson cohomology $BP^*$ as the universal $p$-typical oriented theory (where $p$ is a prime number).
 That lead to the reduction 
of the study of the groups $MU^*(X)\ot\Zp$ 
to the study of $BP^*(X)$ which is structurally a simpler object.
	
The identification $MU_*\cong \LL$ is only a part of the algebraic interpretation
of the Hopf algebroid $(MU_*, MU_*(MU))$. The interpretation of $MU_*(MU)$,
or in other words, of the Landweber-Novikov algebra,
as realising strict isomorphisms between formal group laws allows to identify
 the category of comodules over $(MU_*, MU_*(MU))$ 
with the category of quasi-coherent sheaves over the stack of formal groups (\cite{Nau}).
This category is recognised as a good algebraic approximation\footnote{Note, however, that 
a slightly reformulated {\sl Mahowald uncertainty principle} says that any
algebraic approximation to stable homotopy theory must be infinitely far from correct.}
 to the $MU$-local stable homotopy category,
and the stack point of view introduces geometric intuition to its study.

However, Landweber-Novikov operations which played the main role in considerations above
are stable operations.
It was also Quillen who has introduced certain unstable operations in the picture (\cite{Quil2}),
which allowed him, for example, to prove that the graded group of complex cobordism of a CW-complex as a 
module over the Lazard ring is generated in non-negative degrees.
Note that this result is intrinsically unstable as stable operations are 'independent' 
of the grading, and therefore can not give any restrictions on the grading of elements.

The search for unstable operations in cobordism was continued by many mathematicians,
whose efforts were partially summarised by Boardman, Johnson and Wilson (\cite{BJW}).
They provided an algebraic framework
for modules over some algebra of unstable operations on $BP^*$ 
in an analogy to the stable framework of comodules over the Hopf algebroid $(BP_*, BP_*(BP))$.
They were able to prove that for a finite CW-complex the $BP$-module $BP^*(X)$
has a presentation with relations in positive degree, 
and has a filtration s.t. its factors are either $BP$ or $BP/I(n_i)$ generated in degree greater
or equal to $2\frac{p^{n_i}-1}{p-1}-1$. 

The book in which these results are stated
was published in 1995, accidentally the same year 
when Voevodsky was announcing a strategy of the 
proof of the Bloch-Kato conjecture
in which he assumed existence of so-called motivic algebraic cobordism and motivated
the search for the homotopy theory of algebraic varieties.
Subsequent foundational 
work of Morel and Voevodsky on the $\mathbf{A}^1$-homotopy theory\footnote{also known as
motivic homotopy theory}
gave an opportunity to introduce many topological notions to the world of algebraic geometry.
This process being started almost 20 years ago still carries on without any decline.
In particular, motivic algebraic cobordism was introduced and successfully applied
in the proof of the Milnor Conjecture.

The important feature of motivic homotopy theory
is the existence of two different 'circles', and thus two suspensions,
 which leads to cohomology theories
represented by spectra being bi-graded (and, generally, without periodicity
with respect to any suspension). In particular, the motivic Thom spectra
represents the motivic algebraic cobordism $MGL^{*,*}$,
the infinite Grassmannian represents algebraic K-theory $K_*$ (which is periodic with respect
to a suitable suspension and thus has only one grading)
and motivic Eilenberg-Maclane spaces represent motivic cohomology $H^{*,*}$.
However, for many of the bi-graded cohomology theories there exist a part of them
which is much more amenable to computations,
 and often carries a more geometric and direct description.
These are called small theories, as opposite to big theories.
 For example, the small theory of algebraic K-theory is $K_0$,
the small theory of motivic cohomology is the Chow ring $\oplus_n \CH^n=\oplus_n H^{2n,n}$.

The small theory $\oplus_n MGL^{2n,n}$ of motivic algebraic cobordism
was developed in a seminal paper of Levine and Morel \cite{LevMor}
and is denoted by $\Omega^*=\oplus_n \Omega^n$.
It turns out that this theory carries many features of complex cobordism in topology.
However, one has to remember that the algebraic grading of $\Omega^*$ 
corresponds to even grading of $MU^*$ while cobordisms of odd degree
have no clear interpretation in the world of algebraic varieties.

Algebraic cobordism with rational coefficients are isomorphic
to Chow groups with coefficients in a polynomial algebra over $\QQ$,
and therefore are at least as hard to study as $\CH^*\ot\QQ$.
However, integrally cobordism appear to contain much more information than Chow groups,
as e.g. integral K-theory $K_0$ can be obtained from $\Omega^*$ purely algebraically.
Current investigations of algebraic cobordism include following topics: 
calculations of algebraic cobordism of projective homogeneous varieties (e.g. \cite{Kiri, VishYag, Calmes}),
 study of different "new" oriented theories such as Morava K-theories 
obtained algebraically from algebraic cobordism (\cite{PetSem}, \cite{Sem}, \cite{Sech}) 
as well as many others among which applications to Donaldson-Thomason theory (\cite[Section 14]{LevPad}).
Some of these applications are using information which is known about
$\Omega^*$ as a module over the coefficient ring, and the goal of this paper
is to obtain results on this structure.

The similarity between $MU^*$ and $\Omega^*$
starts with the fact the ring of coefficients
of algebraic cobordism of Levine-Morel can be canonically identified with the Lazard ring.
The construction due to Panin and Smirnov (\cite{PanSmi}) 
together with the fact that 
$\Omega^*$ is a universal oriented cohomology theory
allow to define Landweber-Novikov operations on $\Omega^*$.
For a smooth variety this makes $\Omega^*(X)$ a $(MU_*, MU_*MU)$-comodule,
or, in other words, a quasi-coherent sheaf over the stack of formal group laws.
However, one should keep in mind that even for a sufficiently nice variety $X$, such as a smooth 
quadric, the group $\Omega^*(X)$ may be not finitely generated as $\LL$-module,
and Landweber's structural results 
can not be directly applied as they demand certain finiteness.

The construction of unstable operations is not a simple problem in topology.
It seems that constructing unstable operations in cohomology theories 
of algebraic varieties is even more complicated.
For example, for quite a long time after the appearance of algebraic cobordism of Levine-Morel
there were known no operations on them except for the stable Landweber-Novikov operations. 
Vishik later proved that they generate all stable operations.
The first unstable operations, symmetric operations,
 were constructed by Vishik using an elaborate and elegant construction 
on $\Omega^*$ (\cite{Vish_Symm_2}). These operations proved to be more subtle
 than Landweber-Novikov operations with respect to 2-divisibility phenomena,
and were applied to the questions of rationality of cycles and related issues.

The next major progress in development of operations on small theories
was Vishik's theorem (\cite{Vish1, Vish2}) 
which gives a classification of {\sl all} operations between theories
of the form $\Omega^*\ot_\LL A$ (which are known as free theories).
More precisely, it reduces the classification problem
to a certain algebraic system of equations which depends solely on
the formal group laws of theories involved.
For example, this allowed Vishik to introduce integral Adams operations
on $\Omega^*$, symmetric operations for all primes 
and prove relations of the latter with Quillen-type Steenrod operations on $\Omega^*$.
 The author also used Vishik's fundamental theorem to introduce Chern classes as
operations from algebraic Morava K-theories to the so-called $p^n$-typical theories 
(\cite{Sech} and forthcoming papers).
Symmetric operations and these Chern classes are not known to have analogues in topology,
even though a similar classification result was obtained by Kashiwabara for
operations between generalized cohomology theories on CW-complexes (\cite{Kash}).
However, an important difference between Vishik's and Kashiwabara's results is that 
the latter requires  certain additional properties of generalised cohomology theories
except from being oriented, and these are often not clear to be satisfied. 

At the moment it seems that symmetric operations are the most fundamental
of all operations introduced above. However, in the author's opinion
 they were not yet applied in the literature in their full strength.
The partial goal of this paper is to mitigate this shortcoming
and to obtain structural results of the algebraic cobordism
as a module over the Lazard ring with their help.

In fact, the first application of symmetric operations to the structure of algebraic cobordism
 is due to Vishik, and is the theorem that as $\LL$-module $\Omega^*(X)$ has relations generated 
 in positive degrees. This is a similar statement to a topological result above,
  however Vishik was not aware of it\footnote{Communicated to the author privately.}.
Vishik's result has an important corollary (Prop. \ref{prop:finitely_presented}),
 that for any smooth variety $X$
the module $\Omega^*(X)$ is a union of coherent comodules over $(MU_*, MU_*MU)$,
which allows to apply Landweber's structural results in some situations.

We use symmetric operations to prove the following result, also similar to the known topological statement.

\begin{Th-intro}[Th. \ref{th:filtration}]
The associated graded groups of algebraic cobordism
with respect to the topological filtration $\Omega^*_{(r)}(X)$ 
are unions of finitely presented $\LL$-modules which have a filtration
 such that the corresponding factors are isomorphic
as graded $\LL$-modules to cyclic modules $\LL y$, $\deg y\ge 0$,
or $\LL/I(p,n)x$, $\deg x\ge \frac{p^n-1}{p-1}$.
\end{Th-intro}
 
Our interest in the subject of the structure of cobordism 
comes from the following conjecture formulated by Vishik.
We prove it in Theorem \ref{th:syzygies}.

\begin{Conj}[Vishik, {\cite[Conj. 4.5 p. 981]{Vish_Lazard}}]
Let $X$ be a smooth variety of dimension $d$. 
Then $\Omega^*(X)$ has a free $\LL$-resolution whose $j$-th term has generators concentrated 
in codimensions between $j$ and $d$. 

In particular, the cohomological $\mathbb{L}$-dimension
of the $\mathbb{L}$-module $\Omega^*(X)$ is less or equal to $d$.
\end{Conj}

The fact that generators of $\Omega^*(X)$ lie in non-negative codimensions
is due to Levine and Morel, and is the basis for the so-called generalized degree formula.
As we have already mentioned Vishik has proved that the relations of $\Omega^*(X)$ 
as {\LL-module}
lie in positive codimensions with the use of symmetric operations in \textit{op.cit}.
 In fact, it follows from this statement that $\Omega^*(X)$ is isomorphic as $\LL$-module 
 to $M\ot_{\LL^{\{> -d\}}} \LL$
 where $\LL^{\{> -d\}}$ is a subring of $\LL$ generated by elements 
 of degrees greater than $-d$ and $M$ is a module over this ring which is 
 non-canonically isomorphic to $\ZZ[x_1,\ldots, x_{d-1}]$.
One can deduce from this that $\Omega^*(X)$
has a projective resolution of length $d$ so that the last assertion of the Conjecture is true.

Another application of the structural result above (Th. \ref{th:filtration})
is the description of the algebraic cobordism of a surface 
(the case of a curve was treated by Vishik in {\it op. cit.}).
We also prove a similar statement 
on the structure of $BP^*$ of a smooth variety
of dimension less or equal to $p$ (Th.~\ref{th:BP_pfold}).

\begin{Th-intro}[Th. \ref{th:cobord_surface}]
Let $S$ be a smooth surface.
Then there exist the following exact sequence
$$ 0\rarr \CH^2(S)\ot\LL\rarr \tau^1\Omega^*(S)\rarr \CH^1(S)\ot\LL\rarr 0,$$
where the extension of $\LL$-modules 
is defined by an extension of abelian groups 
$$ 0\rarr \CH^2(S)\cong \tau^2 K_0(S)\rarr \tau^1 K_0(S)\rarr gr_\tau^1 K_0(S)\cong \CH^1(S)\rarr 0.$$

As there is a decomposition $\Omega^*(S)=\LL\cdot 1 \oplus \tau^1\Omega^*(S)$,
this gives a description of algebraic cobordism of a smooth surface as a $\LL$-module in terms
of $K_0(S)$ together with a topological filtration on it.
\end{Th-intro}

We should note that at some point when working with symmetric operations
the author has found a book by Boardman, Johnson and Wilson \cite{BJW} where some unstable operations
on complex cobordism were considered. Surprisingly these operations satisfied
several properties similar to that of symmetric operations. In fact, the author
has read the statement about the filtration on topological $BP$-theory
stated above in this book where it is proved with the use of those unstable operations
 ([Th. 21.12, \textit{op.cit.}]). 
  We do not claim here that those unstable operations are 'the same' as symmetric operations in any sense,
 and to prove the analogous statements in algebraic cobordism
we use the construction of Vishik of symmetric operations.

{\it Outline.} 
In Section \ref{sec:prelim}
we recall the basic facts about algebraic cobordism and operations on them. 
In particular, we show that for a smooth variety $X$ 
the $\LL$-module $\Omega^*(X)$ 
has a structure of $(MU_*, MU_*(MU))$-comodule (Prop. \ref{prop:cobord_comodule})
and use Vishik's results on $\LL$-relations of $\Omega^*$
to show that $\Omega^*(X)$ is an ind-coherent $\LL$-module (Prop. \ref{prop:finitely_presented}).

Section \ref{sec:topfilt_symm} contains the main technical result of the paper.
We investigate the action of symmetric operations on the point (Section \ref{sec:linearity}),
and deduce from it that the restrictions on the structure of $BP^*(X)$
as $BP$-module (Prop. \ref{prop:degree_estimates}).  
Using Landweber's structural results we lift these results to $\Omega^*$
in Th. \ref{th:filtration}.

Section \ref{sec:cobord_surface} is an application of the structural
results to the algebraic cobordism of a surface and the $BP$-theory
of varieties of dimension not greater than $p$.

Section \ref{sec:resolutions} contains 
the proof of the Syzygies Conjecture of Vishik.
It starts with a homological criterion in terms of graded groups $\Tor_*^\LL(\Omega^*(X),\LL/\LL^{<0})$
for the conjecture to be satisfied (Prop. \ref{prop:tor_syzygies}),
and then structural results are applied 
to show that these homological conditions are fulfilled (Th. \ref{th:syzygies}).

\section*{Acknowledgements}

The author has started to work on the structure of algebraic cobordism
in perfect conditions of the Hausdorff Trimester Program "K-theory and Related Fields"
at Hausdorff Research Institute for Mathematics in Bonn in May 2017. 
The author acknowledges the 
support by DFG-Forschergruppe 1920.

The author is extremely grateful to Alexander Vishik 
who has read  drafts of this paper 
finding several mistakes and making a huge number of suggestions to improve the text.

\section{Preliminaries}\label{sec:prelim}

Fix a base field $F$ of characteristic zero.

\subsection{Algebraic cobordism and Brown-Peterson cohomology}\label{sec:def}

In this section we briefly recall the main properties of algebraic cobordism following \cite{LevMor}.

\begin{Def}[{\cite[Def. 1.1.2]{LevMor}}]\label{def:oriented}
An {\sl oriented cohomology theory} $A^*$ is a presheaf of graded rings
on the category of smooth quasi-projective varieties over $F$
supplied with the data of push-forward maps for projective morphisms.
Namely, for each projective morphism of smooth varieties $f:X\rarr Y$ of codimension $d$,
a homomorphism of graded $A^*(X)$-modules $f^A_*:A^*(X)\rarr A^{*+d}(Y)$ is defined.

The structure of push-forwards has to satisfy the following axioms 
(for precise statements see {\sl ibid}):
functoriality for compositions (A1), base change for transversal morphisms (A2),
projection formula, projective bundle theorem (PB),
 $\mathbb{A}^1$-homotopy invariance (EH).
\end{Def}

Each oriented theory $A^*$ can be endowed with Chern classes of vector bundles $c_i^A$
using the classical method due to Grothendieck.
First Chern class allows to associate the formal group law 
$\F_A\in A^*(\Spec F)[[x,y]]$ to the theory $A^*$
so that it satisfies the following equation for every pair of line bundles $L_1, L_2$
over a smooth variety $X$: $c_1^A (L_1\ot L_2)=\F_A(c_1^A(L_1),c_1^A(L_2))$.
Recall that the formal group laws (FGLs) over a ring $A$ are in 1-to-1 correspondence
with the ring morphisms from the Lazard ring $\LL$ to $A$. In particular,
the construction above yields a morphism
 of graded rings $\LL\rarr A^*(\Spec F)$ for all oriented theories $A^*$.

Algebraic cobordism $\Omega^*$ is an oriented theory
which to a smooth variety $X$ associates a free abelian group
generated by (classes of isomorphisms of) projective morphisms to $X$
factored by classical cobordant relations with smooth fibres
and by so-called double-point relations (see \cite[0.4]{LevPad} for details). 
In particular, $\Omega^*(\Spec F)$ is generated by classes
of smooth projective varieties over $F$.
Apart from this, we will only use various universal properties of $\Omega^*$ described below.

\begin{Th}[Levine-Morel, {\cite[Th. 1.2.6]{LevMor}}]\label{th:universality_alg_cobord}
Algebraic cobordism $\Omega^*$ is the universal oriented theory,
i.e. for any oriented theory $A^*$ there exist a unique morphism of
oriented theories $\Omega^*\rarr A^*$.
\end{Th}

Recall that a morphism of oriented theories $A^*\rarr B^*$
is a morphism of presheaves of graded rings which commutes 
with push-forward maps (and with pull-back maps by a definition of a presheaf).

\begin{Th}[Levine-Morel, {[{\it loc.cit.}, Th. 1.2.7]}]
The ring of coefficients of $\Omega^*$ is canonically isomorphic
to the Lazard ring $\LL$
and the formal group law $F_\Omega$ is the universal formal group law.
\end{Th}

The only examples of oriented theories that we will need
are (graded) {\sl free theories} \cite[Rem. 2.4.14]{LevMor}, i.e. $A^*:=\Omega^*\ot_{\LL} A$
where the map of graded rings $\LL \rarr A$ corresponds to a formal group law over $A$,
and the identification $\Omega^*(\Spec F)\cong \LL$ is hidden in the notation.
Chow groups and Grothendieck group of vector bundles $K_0[\beta, \beta^{-1}]$, $\deg \beta=-1$,
 can be obtained this way.
The most important example of a free theory in this paper 
is the so-called {\sl Brown-Peterson cohomology}
which we now describe.

Let $p$ be a prime number.
There exists a notion of a $p$-typical formal group law \cite{Cart},
so that a formal group law over a torsion-free $\Zp$-algebra 
is $p$-typical if and only if its logarithm is of the form $\sum_{i=0}^\infty l_i t^{p^i}$.
There exist a universal $p$-typical formal group law defined
over the graded ring $BP\cong \Zp[v_1,\ldots, v_n,\ldots]$ where $\deg v_i=1-p^i$.
The corresponding free theory is called the {\sl Brown-Peterson cohomology} and denoted as $BP^*$
($p$ is not usually included in the notation).

The map $\LL \rarr BP$ classifying the universal $p$-typical formal group law
has a section localized at $p$: $BP\rarr \LL_{(p)}$ 
which sends $v_j$ to a polynomial generator $x_{p^j-1}$ of $\LL$.
In particular, the map $BP\rarr \LL_{(p)}$ is flat.

The result of Cartier saying that every formal group law over any $\Zp$-algebra
 is canonically 
isomorphic to a $p$-typical formal group law ({\it{}op.cit.}) 
has the following well-known corollary
 for oriented theories. 

\begin{Prop}[{cf. \cite[Th. 4]{Quil}},
  see also Prop. \ref{prop:BP_Omega_algebraic}]\label{prop:from_Omega_to_BP}
There exist a multiplicative projector on $\Omega^*\ot\Zp$
whose image is canonically identified with $BP^*$.
In particular, any additive endo-operation on $\Omega^*\ot\Zp$
canonically restricts to an operation on $BP^*$. 

Moreover, as an $\LL$-module $\Omega^*\ot\Zp(X)$
is isomorphic to $BP^*(X)[t_j, j\neq p^i-1]$
where the structure of $\LL$-module on the latter is defined via 
the map $\theta:\LL_{(p)}\cong BP[t_j, j\neq p^i-1]$.
\end{Prop}

\subsection{Topological filtration on free theories and operations}\label{sec:top_filt}

Define the {\sl topological filtration} (or sometimes referred to also as codimensional filtration) 
on $\Omega^*(X)$
by the formula
$$\tau^i\Omega^*(X):=\cup_{U\subset X: \mathrm{codim}(X\setminus U)\ge i} 
\mathrm{Ker \ }\left(\Omega^*(X)\rarr \Omega^*(U)\right).$$

It turns out that this filtration can be obtained using only the information 
of the structure of $\LL$-module on $\Omega^*(X)$.

\begin{Prop}[Levine-Morel, {\cite[Th. 4.5.7]{LevMor}}]\label{prop:top_omega}
$$ \tau^i \Omega^n=\cup_{m\le n-i} \LL^{m} \Omega^{n-m}. $$
\end{Prop}

For every free theory $A^*$ we will denote the graded sum of factors 
$\oplus_n \tau^r A^n/\tau^{r+1} A^n$ by $A^*_{(r)}$.
 Note that these are graded $A$-modules.
For example, $\Omega^*_{(r)}$ plays a main role in this paper
and several of the main results are formulated 
in terms of the structure of modules $\Omega^*_{(r)}$ over the Lazard ring.
We would like to emphasise that information about these modules 
also gives (partial) information on the structure of $\Omega^*$ as a module over the Lazard ring
due to Prop. \ref{prop:top_omega}.

Levine and Morel show that there exist a surjection of $\LL$-modules \cite[Cor. 4.5.8]{LevMor}

\begin{equation}\label{eq:rho_CH_Omega}
\CH^r(X)\ot_\ZZ \LL \xrarr{\rho} \Omega^*_{(r)}(X),
\end{equation}
and clearly this  induces $A$-linear map 
$\CH^r(X)\ot_\ZZ A \xrarr{\rho_A} A^*_{(r)}(X)$ 
for any free theory $A^*$ with the ring of coefficients $A$.

Any operation $\phi:A^*\rarr B^*$ (which preserves zero) between free theories 
preserves the topological filtration,
and in certain circumstances (e.g. if it is additive) it induces an operation between corresponding
graded factors $\phi_\tau:  A^*_{(\bullet)} \rarr B^*_{(\bullet)}$. 
It can be shown using Vishik's Riemann-Roch theorem for regular embeddings 
and all operations (\cite[Prop. 5.19]{Vish2})
that the operation $\phi_\tau$ can be 'lifted' to an operation on Chow groups 
so that the following diagram is commutative:
\begin{diagram}
\CH^\bullet(X)\ot_\ZZ A &\rTo_{\phi_{\CH}} &\CH^\bullet(X)\ot_\ZZ B \\
\dTo_{\rho_A} & &\dTo_{\rho_B}\\
A^*_{(\bullet)}(X) &\rTo^{\phi_{\tau}} & B^*_{(\bullet)}(X).
\end{diagram}
 
As there are not so many operations between Chow groups with coefficients, 
this description often simplifies the picture vastly.
We will not go into details here, as in all cases when such a description is needed
it will be provided separately (Props. \ref{prop:LN_action_cobord_top}, \ref{prop:symm_action_bp}).

\subsection{Multiplicative operations between oriented theories}

We follow  \cite{PanSmi}, \cite{LevMor} and \cite{Vish1}.

A {\sl multiplicative operation $\Phi$} from an oriented theory $A^*$ 
to $B^*$ is a functor $\Phi:A^*\rarr B^*$ of presheaves of rings 
on the category of smooth varieties over $F$.
Every multiplicative operation $\Phi$ gives rise to a morphism of formal group laws
$(A,\F_A)\xrarr{\phi,\gamma_\Phi} (B,\F_B)$ where
$\phi:A\rarr B$ is the value of $\Phi$ on the point $\Spec F$,
$\gamma_\Phi(t) \in B[[t]]t$ is the series, s.t. 
for every line bundle $L$ over any smooth variety $X$
we have $\Phi(c_1^A(L))=\gamma_\Phi(c_1^B(L))$. 
It follows that $\phi$ and $\gamma_\Phi$ satisfy the following equation:
$\gamma_\Phi (\F_B(x,y))=\phi(\F_A)(\gamma_\Phi(x), \gamma_\Phi(y))$.

If the series $\frac{\gamma_\Phi(t)}{t}$ is invertible,
i.e. $\gamma_\Phi(t)=b_0t+\ldots$ where $b_0\in B^\times$,
then its inverse is called the {\sl Todd genus} $\Td_{\Phi}$.
The standard convention is that one might plug in vector bundle $V$ 
over a smooth variety $X$ in the Todd genus,
so that $\Td_{\Phi}(V)=\prod_{i=1}^{\mathrm{rk}(V)} \Td_{\Phi}(\lambda_i^B)$
 is an element of $B^{\mathrm{rk}(V)}(X)$
 where $\lambda_i^B$ are $B$-roots of $V$.

The following result which is a generalisation of the Grothendieck-Riemann-Roch theorem
was proved by Panin and Smirnov before the algebraic cobordism of Levine-Morel appeared.
However, it is not hard to check that free theories are examples of oriented theories
as defined by Panin-Smirnov.

We will use this result for a calculation with the Quillen-type Steenrod operation
on algebraic cobordism (Lemma \ref{lm:St_vn}).

\begin{Th}[Panin-Smirnov, {\cite[Th. 5.1.4]{PanSmi}}]\label{th:mult_riemann_roch}
Let $\Phi:A^*\rarr B^*$ be a multiplicative operation between oriented theories,
and assume that $\gamma_\Phi$ is invertible, i.e. the Todd genus $\Td_\Phi$ is defined.
Let $f:X\rarr Y$ be a projective morphism of smooth varieties, $\alpha \in A^*(X)$.

Then $f^B_*(\Phi(\alpha) \Td_\Phi(T_X) )= \Phi(f_*^A(\alpha)) \Td_\Phi(T_Y)$.

In particular, if $a\in A$  is $p_*(1_X)$ for some smooth 
projective variety $p:X\rarr \Spec F$,
then $\Phi(a)=\phi(a)= p_* (\Td_\Phi(T_X))$.
\end{Th}

The universality of algebraic cobordism allows to construct
multiplicative operations from it in an efficient manner.
The following will be used to describe Landweber-Novikov 
and Quillen-type Steenrod operations on algebraic cobordism.

\begin{Prop}[Panin-Smirnov+Levine-Morel, see e.g. {\cite[Th. 3.7]{Vish1}}]\label{prop:mult_op_from_cobord}
Let $B^*$ be an oriented theory, let $\gamma = b_0t+\ldots$ be a series in $B[[t]]t$, 
and assume that  $b_0$ is invertible in $B$.

Then there exist a unique multiplicative operation
$\Phi:\Omega^*\rarr B^*$ s.t. $\gamma_\Phi=\gamma$.
\end{Prop}

In particular, this Proposition allows 
to define a multiplicative operation $H:\Omega^*\rarr \CH^*[t_1,t_2,\ldots]$
with $\gamma_H(x)=x+\sum_{i=1}^\infty t_i x^{i+1}$, $\deg t_i=-i$.
Over a point this operation defines an inclusion of rings 
$\LL \hookrightarrow \ZZ[t_1,\ldots]$ (corresponding to a Hurewicz map in the topology of $MU$),
and the image of a class of variety $[X]$
is a polynomial whose coefficients are characteristic numbers of $X$.

Denote by $I(p)$ the ideal in $\Omega^*(\Spec F)\cong \LL$
generated by $p$ and the classes of projective varieties
whose all characteristic numbers are divisible by $p$.
One can show that $I(p)$ is generated
by elements $x_{p^i-1}$ of degree $1-p^i$, one for each $i\ge 1$.
Moreover, the set of these generators $x_{p^i-1}$ can be completed to 
a set of generators of $\LL$.
For $n\colon 1\le n\le \infty$ denote by $I(p,n)=(p, x_{p-1},\ldots, x_{p^{n-1}-1})$ the subideal of $I(p)$
generated by the elements of dimension $\le p^{n-1}-1$, by $I(p,1)$ denote$(p)$,
and denote also $I(p,\infty)=I(p)$.

Note that the image of element $x_{p^i-1}$ in $BP$ is denoted by $v_i$, 
and $v_0$ often denotes $p\in BP$. These elements generate $BP$ freely,
i.e. $BP\cong \Zp[v_1,\ldots, v_n,\ldots]$,
and the splitting of the projector to $BP^*$ (Prop. \ref{prop:from_Omega_to_BP})
sends $v_i$ to $x_{p^i-1}$ in $\LL_{(p)}$.
We will denote by $I(n)$\footnote{
Thus, we have very abusing notations as e.g. $I(3)$ might refer
 to an ideal in $\LL$ as well as to an ideal in that $BP$ which is a $\ZZ_{(3)}$-algebra.
 We hope this does not lead to any misunderstandings throughout the paper,
 as it should always be clear whether we are working with $\Omega^*$ or $BP^*$.}
  the ideal $(v_0,v_1,\ldots, v_{n-1})$ in $BP$ for $n\ge 1$,
and by $I(\infty)=\cup_{n} I(n)$. Therefore, the map $BP\rarr \LL_{(p)}$
sends $I(n)$ to $I(p,n)$, and $I(n)\ot_{BP} \LL_{(p)}=I(p,n)$.

\subsection{Landweber-Novikov operations on algebraic cobordism}

The goal of this section is to recall some results on the structure of
$(MU_*,MU_*(MU))$-comodules and explain why they can be applied to the algebraic cobordism.

By definition free theories are in 1-to-1 correspondence with formal group laws.
However, many of the formal group laws are isomorphic,
and corresponding theories turn out to be isomorphic as presheaves of rings,
or, in other words, there exist 
an invertible multiplicative operation between these theories.
Let us explain this in more detail for the case of so-called strict isomorphisms.

A morphism $(\phi, \gamma):(A,F_1)\rarr (A,F_2)$ 
between formal group laws over the ring $A$
is called a {\sl strict isomorphism} if $\phi=\id$,
and $\gamma\in t+A[[t]]t^2$.
Given a strict isomorphism,
the Panin-Smirnov reorientation of an oriented theory yields an invertible multiplicative operation
from a presheaf of rings $A^*$ to itself.

Thus, from the point of view of {\it orientable} cohomology theories
(which are presheaves of rings which can be endowed with the structure of push-forwards, 
but this structure is not fixed), the algebraic object which should appear in their classification
is not a formal group law, 
but a formal group law up to an isomorphism, or at least up to a strict isomorphism. 

FGLs over a ring form a set and are classified by the maps from the Lazard ring,
however FGLs with a strict isomorphism (sIso) over a ring form a groupoid 
as one might compose strict isomorphisms.
The algebraic structure which 'classifies groupoids' is called a Hopf algebroid (\cite[A1.1.1]{Rav}),
and, in particular, the groupoid (FGL, sIso) is classified
by the Hopf algebroid $(\LL, \LL B)$ often denoted as $(MU_*, MU_*(MU))$
 in the topological literature
(see e.g. \cite[A2.1.16, 4.1.11]{Rav}). 
Here $\LL B=\LL \ot_\ZZ \ZZ[b_1, b_2,\ldots]$, $B=\ZZ[b_1, b_2,\ldots]$.

The total Landweber-Novikov operation in algebraic cobordism
is defined as the universal strict isomorphism, and thus specialises 
to all strict isomorphisms between any pair of free theories.
This allows to define the structure of a right $(\LL, \LL B)$-comodule on $\Omega^*(X)$
for every smooth variety $X$.

\begin{Def}
Let $\gamma_{L-N}$ be the power series over
the graded ring $B=\ZZ[b_1,b_2,\ldots]\subset \LL B$
given by the formula: $\gamma_{L-N}(t)=t+\sum_{i\ge 1} b_it^{i+1}$.

The {\sl total Landweber-Novikov operation} is the unique operation 
(see Prop. \ref{prop:mult_op_from_cobord})
corresponding to $\gamma_{L-N}$:
$$ S_{L-N}^{tot}: \Omega^* \rarr \Omega^*\ot_{\LL} \LL B.$$
\end{Def}

This operation induces an operation on the graded algebraic cobordisms $\Omega^*_{(r)}$,
and it can be easily described in terms of the cycles that generate this $\LL$-module
(see Section \ref{sec:top_filt}).

\begin{Prop}\label{prop:LN_action_cobord_top}
Let $X$ be a smooth variety, and denote by $z$ the image $\rho(Z)$ of some element $Z\in \CH^r(X)$
 in the group $\Omega^*_{(r)}$. Let $\lambda\in\LL$.
 Then $S^{tot}_{L-N}(\lambda z)= S^{tot}_{L-N}(\lambda)z$ in $\Omega^*_{(r)}(X)$. 
\end{Prop}
\begin{proof} 
As we are working modulo the $r+1$-th part of the topological filtration 
we may assume that $Z$ is represented by a smooth subvariety $i:Z\hookrightarrow X$.
Then Riemann-Roch Theorem \ref{th:mult_riemann_roch}
shows that $S^{tot}_{L-N}(i_* 1_Z)$ equals $i_* 1_Z$ modulo the $r+1$-th part 
of the topological filtration.
The claim follows from the multiplicativity of $S_{L-N}^{tot}$.
\end{proof}

It might seem that the following proposition is more or less tautological.
This is partly true, but 
in the sketch of a proof we actually check that $\Omega^*(X)$
is a right graded comodule over the Hopf algebroid $(MU_*, MU_*(MU))$
via the identification of the latter with $(\LL, \LL B)$
made explicit in \cite[A2.1.16]{Rav}.
This will allow us to apply structural results of Landweber to algebraic cobordism below.

\begin{Prop}\label{prop:cobord_comodule}
The total Landweber-Novikov operation
defines on $\Omega^*$ the structure of 
presheaf of right graded comodules over the Hopf algebroid $(\LL, \LL B)$
on the category of smooth varieties.
\end{Prop}
\begin{proof}
By definition a right graded $(\LL, \LL B)$-comodule is a graded $\LL$-module $M$
together with a map $\psi:M\rarr M\ot_\LL \LL B$ of $\LL$-modules
where the structure of $\LL$-module on $M\ot_\LL \LL B$ is defined via the map 
$S:\LL \rarr \LL B$ corresponding to the universal formal group law over $\LL$
twisted by the series $\gamma_{L-N}(t)$. Note that $S$ is equal
to the action of the total Landweber-Novikov operation over a point.
The map $\psi$ has to satisfy the following conditions:
 1) (counitarity) $(\id_M\ot \epsilon)\circ \psi = \id_M$ and 
 2) (coassociativity) $(\id_M\ot \Delta)\circ \psi = (\psi\ot \id_{\LL B} )\circ \psi$.
Here $\epsilon$ and $\Delta$ denote two of the five structural maps of the Hopf algebroid.

Thus,  in both 1) and 2) 
we need to check that two operations from $\Omega^*$ to $\Omega^*$ 
or to $\Omega^*\ot_\LL \LL B\ot_\LL \LL B$
are equal.
Note that in both cases these operations are multiplicative,
and therefore by Prop. \ref{prop:mult_op_from_cobord}
it is enough to check that the series defining corresponding morphisms of formal group laws are the same.
 Note that the composition of morphisms of formal group laws $(A,F_A)\xrarr{(f,\gamma_f)} (B,F_B)$, 
$(B,F_B)\xrarr{(g,\gamma_g)} (C,F_C)$ is defined by a pair $(g\circ f, g(\gamma_f)(\gamma_g))$.

For 1) it is clear, as $\epsilon:\LL B\rarr \LL$ is defined on $b_i$ as zero for all $i\ge 0$,
and therefore $\epsilon (\gamma)(t) = t$.


For 2) we have to compare the series 
$(S\ot_\ZZ \id_{\ZZ B})(\gamma_2)(\gamma_1(t))=\sum_{j\ge 0} (\sum_{i\ge 0} b_i t^{i+1})^{j+1} \ot b_j$
where $\gamma_i = j_i(\gamma)$, $i=1,2$, $j_1,j_2:\LL B\rarr \LL B\ot_\LL \LL B$ are $(\id \ot 1)$, $(1\ot \id)$, respectively,
with the series $\Delta(\gamma (t))$.
The morphism $\Delta:\LL B \rarr \LL B \ot \LL B$
is defined by $\gamma(t) \mapsto \gamma_2(\gamma_1(t))$.
Thus, the series defining $(\id_M\ot \Delta)\circ \psi$ is also $\gamma_2(\gamma_1(t))$
(note that the map $S$ does not act on coefficients of the~series~$\gamma_2$).
%
\end{proof}

In particular, $\Omega^*(\Spec F)=\LL$ is a $(\LL, \LL B)$-comodule,
and it is easy to see that its subcomodules are precisely ideals which are invariant 
w.r.to Landweber-Novikov operations. We will call such ideals
{\sl invariant} throughout the article. 

The reason why we had to introduce the Hopf algebroid into the game
is that the foundational results about the action of Landweber-Novikov operations
in topology are written in this language. 
We will use mainly the following results.

\begin{Th}[Landweber]\label{th:invariant_ideals}
\leavevmode
\begin{enumerate}
\item[$\LL$)]  {\cite[Th. 2.7]{Land3}}
If $I$ is an invariant prime ideal  in $\LL$,
 then $I$ is one of the ideals $0$, $I(p,n)$ for some prime number $p$, $1\le n\le \infty$.
 
\item[$BP$)] {\cite[Th. $2.2_{BP}$]{Land}}
If $I$ is an invariant prime ideal  in $BP$,
 then $I$ is one of the ideals $0$, $I(n)$, $1\le n\le \infty$.
\end{enumerate}
\end{Th}

\begin{Th}[Landweber, {cf. \cite[Prop. 3.4]{Land3}}]\label{th:radical-invariant-ideal}
Let $I\subset BP$ be an invariant finitely generated ideal.

Then the radical of $I$ is an invariant finitely generated prime ideal in $BP$,
i.e. one of the form $I(n)$ for some $n$.
\end{Th}
\begin{proof}
Indeed, it is proved in the reference for the case $I\subset \LL$ that 
$I$ is an intersection of a finite number of invariant finitely generated primary ideals $Q_i$
s.t. radicals $\sqrt{Q_i}$ are distinct invariant finitely generated prime ideals.
The proof in the case of $I\subset BP$ is similar.

The claim now follows, since a radical of intersection is an intersection of radicals,
and by the structure of the invariant prime ideals in $BP$ (Th. \ref{th:invariant_ideals}, $BP$)
 their intersection is one of the ideals $0$, $I(n)$ for a finite $n$. 
\end{proof}

\begin{Cr}\label{cr:radical}
If $I\subset BP$ is an invariant finitely generated ideal
s.t. $(p,v_1,\ldots, v_{n-1})=I(n)\subset I$ for a finite $n$,
 $v_{n}^i \notin I$ for all $i\ge 1$,
then $I=I(n)$.
\end{Cr}
\begin{proof}
The radical ideal of $I$ is $I(m)$ which contains $I(n)$ and does not contain $v_{n}$,
therefore it is $I(n)$, and therefore it is equal to $I$.
\end{proof}

The following result is a stable version 
for the structure of cobordism as a module over the Lazard ring,
we will prove a stronger and essentially unstable version 
in Th. \ref{th:filtration}.

\begin{Th}[Landweber, {\cite[Th. 3.3]{Land2}}, {\cite[Th. 2.2, 2.3]{Land}}]\label{th:landweber_filtration}
Let $M$ be a graded comodule over $(\LL,\LL B)$ which 
is finitely presented as $\LL$-module.

Then $M$ has a filtration $M=M_0\supset M_1\supset \ldots \supset M_s=0$
by $(\LL,\LL B)$-comodules,
s.t. for every $i$ the $\LL$-module $M_i/M_{i+1}$ is isomorphic to $\LL/I(p_i,n_i)$ or $\LL$ 
after a shift of grading
where $p_i$ is a prime number and $n_i$ is a positive integer.
\end{Th}

The following proposition is an algebraic version of Prop. \ref{prop:from_Omega_to_BP}.

\begin{Prop}\label{prop:BP_Omega_algebraic}
Let $M$ be a $(\LL, \LL B)$-comodule.

Then $M_{(p)}:=M\ot_\LL \LL_{(p)}$ is isomorphic 
as $\LL$-module to $(M\ot_{\LL} BP)\ot_{BP} \LL_{(p)}$,
where the map $\LL\rarr BP$ corresponds to the universal $p$-typical group law,
and the map $g:BP\rarr \LL_{(p)}$ corresponds to the $p$-typical
formal group law over $\LL_{(p)}$ which is strictly isomorphic to the universal formal group law
(such map exists by a result of Cartier).
\end{Prop}
\begin{proof}
Let $f_1,f_2:\LL \rarr A$ be two ring maps corresponding to formal group laws $F_1, F_2$
over $A$ respectively, and suppose given a strict isomorphism from $F_1$ to $F_2$ over $A$.
At first, we are going to show that for every $(\LL, \LL B)$-comodule $M$
there exist a canonical isomorphism $M\ot_{\LL, f_1} A\cong M\ot_{\LL, f_2} A$ of $A$-modules.
(Since the Hopf algebroid $(\LL, \LL B)$ represents the groupoid of formal group laws with strict isomorphisms, this statement is expected, however, the proof, perhaps, is not so enlightening.)

If $\gamma\in A[[x]]$ is a strict isomorphism between $F_1$ and $F_2$,
then denote by $\phi_{f_2,\gamma}:\LL B\rarr A$ a map of rings which sends $b_i$
to the coefficient of $t^{i+1}$ of $\gamma(t)$ for $i\ge 1$ and sends $\lambda \in \LL$ to $f_2(\lambda)$.
Then the map $M\ot_\LL \LL B \xrarr{\id \ot \phi} M\ot_{\LL, f_2} A$ is a well-defined map
of $\LL$-modules. 
Note also that the composition $\LL\xrarr{S} \LL B\xrarr{\phi_{f_2,\gamma}} A$
corresponds to a formal group law over $A$ isomorphic to $F_2$ via~$\gamma$, i.e. to $F_1$,
and the composition is equal to $f_1$.

We claim that there exist a unique $A$-linear map $h_\gamma$ which fits into the
 following commutative diagram:

\begin{diagram}
M &\rTo^{\psi} & M\ot_\LL \LL B \\
\dTo & &\dTo_{\id\ot\phi_{f_2,\gamma}} \\
M\ot_{\LL, f_1} A &\rTo^{h_{\gamma}} & M\ot_{\LL,f_2} A.
\end{diagram}

The uniqueness is clear, since $h_\gamma(m\ot a)$ has 
to be equal to $(\id \ot \phi_{f_2,\gamma})(\psi (m))\cdot a$.
We need to check that this formula is well defined,
i.e. $h_\gamma(\lambda m \ot a)=(\id \ot \phi_{f_2,\gamma})(\psi (\lambda m))\cdot a$
is equal to $h_\gamma(m\ot f_1(\lambda) a)=(\id \ot \phi_{f_2,\gamma})(\psi (m))\cdot f_1(\lambda) a$.
Recall that $\psi$ is a morphism of $\LL$-modules where 
the structure of the $\LL$-module on $M\ot_\LL \LL B$ is defined via the map $S:\LL \rarr \LL B$,
i.e. $\psi (\lambda m) = \psi(m)\cdot S(\lambda)$.
The map $\id \ot \phi_{f_2,\gamma}$ sends 
$\psi(m)\cdot S(\lambda)$ 
to $(\id \ot \phi_{f_2,\gamma})(\psi(m))\cdot (\phi_{f_2,\gamma}\circ S)(\lambda)$,
and, thus, the claim follows since $\phi_{f_2,\gamma}\circ S:\LL \rarr A$ 
is equal to $f_1$.

We also claim that the map $h_\gamma$ is functorial for compositions 
of strict isomorphisms of formal group laws,
i.e. if $\eta \in A[[x]]$ defines a strict isomorphism between $F_2$ and $F_3$,
then $h_{\eta} \circ h_{\gamma} = h_{\gamma(\eta)}: M\ot_{\LL, f_1} A \rarr M\ot_{\LL,f_3} A$.
Moreover, the identity isomorphism yields an identity morphism $h_{x}=\id_{M\ot_{\LL} A}$.
The last claim is straight-forward, since in this case the map $\LL B\rarr A$ factors through the map
 $\LL B\xrarr{\epsilon\colon b_i \mapsto 0} \LL$
and $\LL \xrarr{f_1=f_2} A$, however, $(\id \ot \epsilon)\circ \psi: M\rarr M$ is the identity
by the definition of a comodule. 

Following the construction above the composition $h_\eta \circ h_\gamma$
can be computed as the composition 
$$M\xrarr{\psi} M\ot_\LL \LL B\xrarr{\psi\ot \id} M\ot_\LL \LL B \ot_\LL \LL B 
\xrarr {\id \ot \phi_{f_3, \eta} \ot \phi_{f_2, \gamma}} M\ot_{\LL, f_3} A.$$
By definition of a comodule we may replace the map $\psi \ot \id$
by $\id\ot \Delta$ above, and then the claim follows 
since the map $\phi_{f_3, \eta} \ot \phi_{f_2, \gamma}\circ \Delta: \LL B\rarr A$
is equal to $\phi_{f_3, \gamma(\eta)}$ more or less by definition of $\Delta$.

Each map $h_\gamma$ constructed
 above is an isomorphism, since each strict isomorphism of formal group laws also has an inverse strict 
 isomorphism.
Thus, an isomorphism $M\ot_{\LL, f_1} A\cong M\ot_{\LL, f_2} A$ of $A$-modules
as claimed in the beginning of the proof is constructed.

We have two ring maps from $\LL$ to $\LL_{(p)}$.
One is the 'identical', corresponding to a universal formal group law,
and the other is the composition $f_{p-typ}:\LL\rarr BP \rarr \LL_{(p)}$
corresponding to a universal $p$-typical formal group law (and depends on the choice of $g$).
It follows from a theorem of Cartier that corresponding two formal group laws over $\LL_{(p)}$
are strictly isomorphic,
and therefore by discussion above
we have an isomorphism $h:M_{(p)}\cong M\ot_{\LL,f_{p-typ}} {\LL_{(p)}}$.
However, the latter module is easily identified with $(M\ot_{\LL} BP)\ot_{BP} \LL_{(p)}$
by the definition of $f_{p-typ}$.
\end{proof}

\subsection{Steenrod and Symmetric operations on algebraic cobordism and $BP$}
Let us fix a set of integers $\bar{i}=\{i_j|0<j<p\}$ of 
all representatives of non-zero numbers modulo $p$. Denote also $\mathbf{i}=\prod_{j=1}^{p-1} i_j$.
Then by Prop. \ref{prop:mult_op_from_cobord} there exist 
a multiplicative operation 
$$ St(\bar{i}):\Omega^*\rarr \Omega^*[\mathbf{i}^{-1}][[t]][t^{-1}]$$
defined by the power series $\gamma_{St(\bar{i})}(x)=\gamma_{St}(x)=x\prod_{j=1}^{p-1}(x+_\Omega i_j\cdot_\Omega t)$,
which is called {\sl the Quillen-type Steenrod operation}.
We will often drop $\bar{i}$ from
the notation of $St$.

Note that $p\nmid \mathbf{i}$, 
and by Prop. \ref{prop:from_Omega_to_BP}
the operation $St$ restricts to a multiplicative operation from $BP^*$
to $BP^*[[t]][t^{-1}]$ which we also denote $St$. 

The following lemma is a strengthening of a particular case of {\cite[Prop. 3.1]{Vish_Lazard}}.
It will be needed in later sections. 

\begin{Lm}\label{lm:St_vn}
For $v_n\in BP$, $n\ge 1$, we have the identity in $BP[[t]][t^{-1}]$:
$$ St(v_n)\equiv v_nt^{-(p-1)(p^n-1)} \mod I(n).$$
\end{Lm}
\begin{proof}
Let $X$ be a smooth projective variety, which class in $BP=BP^*(\Spec F)$ 
is equal to  $v_n$.
By the Riemann-Roch formula (Th. \ref{th:mult_riemann_roch}),
we have $St(v_n)=\pi_*(\Td_{St}(X))$
where $\pi:X\rarr Spec(k)$ is the structural morphism.

Plugging $\gamma_{St}$ in the definition of the Todd genus 
we obtain  $\Td_{St}(x)=\prod_{j=1}^{p-1} \frac{1}{x+_\Omega [i_j]\cdot_\Omega t}$,
and $\Td_{St}(X) = \prod_{i=1}^{\dim X} \Td_{St}(\lambda_i)$ where $\lambda_i$ 
are the $BP$-roots of the tangent vector bundle of $X$.
Thus,  $\Td_{St}(X)$ is the sum of symmetric polynomials in variables $\lambda_i$
with coefficients in $BP[[t]][t^{-1}]$. 

Let $P$ be a symmetric homogeneous polynomial of $\lambda_i$,
i.e. an element of $BP^{\deg P}(X)$.

\begin{itemize}
\item If $\deg P>p^n-1$, then $P=0$, as $\dim X=p^n-1$, and $BP^{>\dim X}(X)=0$.
\item If $\deg P=p^n-1$,
 then $\pi_* P$ is a characteristic number of $X$,
and therefore it is divisible by $p$, in particular, $\pi_* P\in I(n)$.
\item If $0< \deg P <p^n-1$,
 then $\pi_* P \in I(n)$,
because all elements of degrees from $-p^n+2$ to 1 in $BP$ lie in $I(n)$.
\end{itemize}

Therefore, as $\pi_*$ is $BP$-linear, $\pi_*$ applied to positive degree part of $\Td_{St}(X)$ 
lies in $I(n)$. Thus, $\pi_* (\Td_{St}(X))$ modulo $I(n)$ is the same
as the push-forward of 0-degree part of the $\Td_{St}(X)$,
i.e. $\pi_* (\prod_{i=1}^{p^n-1} \prod_{j=1}^{p-1} \frac{1}{i_jt})$
which is equal to $\frac{1}{\mathbf{i}^{p^n-1}} v_nt^{-(p-1)(p^n-1)}$.
As $\mathbf{i}\equiv -1 \mod p$, and $(-1)^{p^n-1}\equiv 1 \mod p$, the claim follows.
\end{proof}

\begin{Th}[Vishik, {\cite[Th. 7.1]{Vish_Symm_all}}]
There exist a unique operation 
$\Phi(\bar{i}):\Omega^*\rarr \Omega^*[\mathbf{i}^{-1}][t^{-1}]$, 
such that 
\begin{equation}\label{eq:symm}
(\square^p-St(\bar{i})-\frac{p\cdot_\Omega t}{t}\Phi(\bar{i})):\Omega^*
\rarr \Omega^*[\mathbf{i}^{-1}][[t]]t.
\end{equation}
\end{Th}

It is convenient to use 'slices' of {\sl the} symmetric operation $\Phi(\bar{i})$
defined as coefficient of monomial $t^{-n}$ for some $n\ge 0$.
We will denote this operation as $\Phi_{-n}(\bar{i})=\Phi_{-n}$,
and call these symmetric operations.

Even though, $\Phi(\bar{i})$ is not additive, it can be shown
to restrict to the operation $\Omega^*_{(p)}\rarr \Omega^*_{(p)}[t^{-1}]$,
and, thus, by Prop. \ref{prop:from_Omega_to_BP} to $BP^*\rarr BP^*[t^{-1}]$ as well (see \cite[p. 977]{Vish_Lazard}).
In subsequent sections we will work only with this $p$-local version
of symmetric operations defined on $BP^*$ denoting it, and its slices by the same letters.

Moreover, $\Phi$ restricts to the graded factors of the topological filtration of $BP^*$,
 and the action on the graded factors has a particularly easy description\footnote{ The
 fact that the symmetric operation induces an operation on the graded factor
 of the topological filtration is not straight-forward as the operation is not additive.
 However, its restriction to $BP^*_{(r)}$ for $r>0$ 
 turns out to be additive as follows from Vishik's description.}

\begin{Prop}[Vishik, {\cite[Prop. 7.14]{Vish_Symm_all}}]\label{prop:symm_action_bp}
Let $X$ be a smooth variety, $\lambda\in BP$, $r>0$, 
$Z\in \CH^r(X)$, and denote by $z$ the image of $Z$ in the group $BP^r_{(r)}$ under the map $\rho_{BP}$ 
(see Section \ref{sec:top_filt}).

Then the following identity holds in $BP^*_{(r)}(X)[t^{-1}]$:
$$ \Phi(\lambda z) = \mathbf{i}^r\cdot t^{r(p-1)}\cdot \Phi_{\le -r(p-1)}(\lambda) \cdot z.$$
\end{Prop}

Recall that Levine and Morel proved that generators of $\LL$-module $\Omega^*_{(r)}(X)$
lie in degree $r$ (\cite[Cor. 4.5.8]{LevMor}, see (\ref{eq:rho_CH_Omega}) above).
It follows that $\Omega^*(X)$ has $\LL$-generators in degrees between 0 and $\dim X$.
One of the most fascinating applications of symmetric operations is the following result.

\begin{Th}[Vishik, {\cite[Th. 4.1, 4.3]{Vish_Lazard}}]\label{th:relations_positive}
Let $X$ be a smooth variety of dimension $d$.
Then $\Omega^*(X)$ (or $\Omega^*_{(r)}(X)$ for some $r:0\le r\le d$)
has a free presentation as $\LL$-module $R\rarr F\rarr \Omega^*(X)$ (resp., $R\rarr F\rarr \Omega_{(r)}^*(X)$)
where $F$ is generated in degrees between 0 and $d$ (resp., in degree $r$)
and $R$ is generated in degrees between 1 and $d$ (resp., between 1 and $r$).
\end{Th}

\begin{Cr}\label{cr:CH1_Vishik}
Let $X$ be a smooth variety, let $A^*$ be a free theory.

Then the map $\rho_A:\CH^1(X)\ot A \rarr A^*_{(1)}(X)$ is an isomorphism.
\end{Cr}
\begin{proof}
Presheaf $\Omega^*$ is generically constant (\cite[Def. 4.4.1, Cor. 4.4.3]{LevMor}),
and, thus, the topological filtration on it splits in the first step:
$\Omega^*=\LL\oplus \tau^1\Omega^*$. Tensoring by $A$ over $\LL$ we obtain an isomorphism 
$\Omega^*_{(1)}\ot_{\LL} A = A^*_{(1)}$, since the canonical map $\tau^2 \Omega^*\ot_\LL A\rarr \tau^2 A^*$ is a surjection.
Therefore the claim for a free theory $A^*$ follows by a base change from $\Omega^*$.

Let us prove the statement for $A^*=\Omega^*$.
The restriction of $\rho$ to $\CH^1(X)$ is injective,
as there is a surjection $\Omega^*_{(1)}(X)\rarr \CH^1(X)$
and its composition with $\rho$ is the identity map.
However, by Vishik's theorem the kernel of $\rho$ is generated in degree 1,
and therefore is zero.
\end{proof}

\subsection{Ind-coherence of algebraic cobordism over the Lazard ring}\label{sec:indcoh}

Let $R$ be a ring. Recall that a $R$-module is called {\sl coherent}
 if it is finitely generated and all its finitely generated submodules (including the whole module)
are finitely presented.
Let us say that a $R$-module is {\sl ind-coherent}
if it is a union of coherent modules,
or, which is the same, if every finitely generated submodule is finitely presented.

\begin{Prop}\label{prop:finitely_presented}
Modules over the Lazard ring $\Omega^*(X)$, $\Omega^*_{(r)}(X)$ are ind-coherent $\LL$-modules.
The $BP$-version also holds: $BP$-modules $BP^*(X)$, $BP^*_{(r)}(X)$ are ind-coherent.

In particular, the annihilator of any element of these modules is finitely generated.
\end{Prop}
\begin{proof}
To deal with all the cases simultaneously let $A=\oplus_{i\le 0} A^i$ be a graded ring
 which is polynomial over the noetherian ring $A^0$ s.t. $A^i$ is a free $A^0$-module of finite rank.
Denote by $A^{\ge -D}$ the $A^0$-subalgebra of $A$ generated by elements in degrees greater or equal to $-D$.
It follows from the assumptions on $A$ that $A^{\ge -D}$ is a noetherian ring, and $A$ is flat over $A^{\ge -D}$ for any $D\ge 0$.
Let $M=\oplus_{i\le d} M^i$ be a graded module over $A$. 

Assume that there exist a presentation of $M$: $A(R)\xrarr{f} A(G)\rarr M\rarr 0$,
where $A(G)$ is a free graded $A$-module generated by a set $G$ in degrees between 0 and $d$,
$A(R)$ is a free graded module generated by a set $R$ in degrees between 1 and $d$.
Our goal is to prove that any finitely generated submodule $N$ of $M$ is finitely presented,
which then will prove the statement as the modules $\Omega^*(X)$ and alike
 have such presentations as above by Th. \ref{th:relations_positive}.

The following lemma will allow us to reduce everything to the 
noetherian world of $A^{\ge -D}$-modules.

\begin{Lm}\label{lm:essentially_noetherian}
Let $A$ be as above, and let $K$ be a graded $A$-module.

Assume that $K\ot_A A/A^{<0}$ and $\Tor_1^A(K,A/A^{<0})$ 
are concentrated in degrees between $a$ and $a+D$ for some $a\in \ZZ$, $D\ge 0$.

Then there exist a module $\tilde{K}$ over the ring $A^{\ge -D}$ 
s.t. $K=\tilde{K}\ot_{A^{\ge -D}} A$.
\end{Lm}
\begin{proof}[Proof of the Lemma.]
It follows from the assumptions that there exist a presentation of $K$:
 $A(R)\xrarr{f} A(F)\rarr K\rarr 0$,
where $A(R)$, $A(F)$ are free $A$-modules generated in degrees between $a$ and $a+D$
by sets $R$ and $F$.

Note that  the morphism $f$ is defined over $A^{\ge -D}$ as 
it sends a generator of relations $R$ to an element which is a linear combination of generators of $F$
with some coefficients. For this element to have degree between $a$ and $a+D$
the coefficients need to have degree no less than $-D$.

Consider a module $\tilde{K}$ defined over $A^{\ge -D}$
by a presentation $A^{\ge -D}(R) \xrarr{f} A^{\ge -D}(F) \rarr \tilde{K}\rarr 0$
 with the same sets of generators and relations as $K$,
 and the morphism $f$ defined in the same way.
Tensoring this sequence by $A$ over $A^{\ge -D}$
we obtain the presentation of $K$, and therefore $K=\tilde{K}\ot_{A^{\ge -D}} A$.
\end{proof}

Let $N$ be a finitely generated submodule of $M$, consider
an exact sequence of $A$-modules: $0\rarr N\rarr M\rarr M/N\rarr 0$.
Applying $\ot_A A/A^{<0}$ to it we obtain the long exact sequence of $A_0$-modules,
from which we emphasise two pieces.
First, from the exact sequence $M\ot_A A/A^{<0}\rarr M/N \ot_A A/A^{<0} \rarr 0$
it follows that $M/N \ot_A A/A^{<0}$ is concentrated in degrees between $0$ and $d$.
Second, from the exact sequence $\Tor_1^A(M,A/A^{<0})\rarr \Tor_1^A(M/N,A/A^{<0}) \rarr N\ot_A A/A^{<0}$
it follows that $\Tor_1^A(M/N,A/A^{<0})$ is concentrated in degrees between $min(m,0)$ and $d$
where $m$ is the minimal degree of a finite set of generators of $N$.

Applying Lemma \ref{lm:essentially_noetherian} we see that
$M=\tilde{M}\ot_{A^{\ge -D}} A$, $M/N=\tilde{M/N}\ot_{A^{\ge -D}} A$ 
for some $D\ge 0$ and $A^{\ge -D}$-modules $\tilde{M}$, $\tilde{M/N}$.
We claim that the canonical morphism $M\rarr M/N$ also comes from 
the morphism $\tilde{\pi}:\tilde{M}\rarr \tilde{M/N}$ 
by the arguments of Lemma \ref{lm:essentially_noetherian}.

Denoting by $\tilde{N}$ the kernel of $\tilde{\pi}$
it is easy to see that $N=\tilde{N} \ot_{A^{\ge -D}} A$
as $A$ is flat over ${A^{\ge -D}}$.
As $\tilde{N}$ is a finitely generated module over the noetherian ring,
it is therefore finitely presented. 
Tensoring the finite presentation of $\tilde{N}$
by $A$ we get the finite presentation of $N$.
\end{proof}
\begin{Rk}
To compare the previous Corollary with results in Topology note that 
for a finite CW-complex $X$ the $\LL$-module $MU^*(X)$ is always finitely generated,
and the fact that it is coherent was proved in \cite[Th. 1.3]{ConSmi}.

In algebraic geometry the $\LL$-module $\Omega^*(X)$ does not have to be finitely generated even
for a geometrically cellular variety $X$ 
(which is a notion presumably quite close to that of a CW-complex).
More precisely, there are examples of  quadrics
 for which $\CH^*$ is not a finitely generated abelian group,
and therefore $\Omega^*$ is not a finitely generated $\LL$-module.
\end{Rk}

\begin{Rk}\label{rem:indcoh_abelian_cat}
Note that ind-coherent modules form a full abelian subcategory in the category of all modules,
kernels and cokernels being computed in the category of all modules.
Thus, whenever we have a map $f:X\rarr Y$ of smooth varieties,
the kernel, image and cokernel of $f^*:\Omega^*(Y)\rarr \Omega^*(X)$ 
are ind-coherent modules which also are $(\LL,\LL B)$-comodules.

In particular, if $X$ is geometrically cellular variety,
 the $\LL\mbox{-}$submodule of rational elements
of a free finitely generated $\LL$-module $\Omega^*(\bar{X})$
is coherent.
\end{Rk}

\section{Symmetric operations and the structure of cobordism}\label{sec:topfilt_symm}

The goal of this section is to use symmetric operations
to prove structural results on $BP^*_{(r)}$ and $\Omega^*_{(r)}$ 
as $BP$- and $\LL$-modules, respectively. 
Let us briefly explain how it is done.

The Landweber's structural result in the case of a coherent cyclic $BP$-module 
with the action of Landweber-Novikov operations (i.e. $(BP, BP_*BP)$-comodule)
can be shown roughly as follows (cf. \cite[proof of Th. 20.11]{BJW}).
 Denote by $x$ the generator of this module,
and without loss of generality assume that $x$ is $p$-torsion, which is equivalent
by Th.~\ref{th:radical-invariant-ideal} to the claim that the module is not free.
Then one proves that there exist $u\in BP$ s.t. $\Ann(ux)=I(n)$ for some $n$
and $(ux)$ is invariant with respect to Landweber-Novikov operations. It allows
to continue the process by induction considering at the next step the module $BP\cdot x/ (u)x$,
even though the question of finiteness of this process is not straight-forward
since $BP$ is not noetherian. We will explain the termination of this process 
in the case of our interest separately.
In order to achieve such conditions on the element $u$,
one chooses it to be equal to $p^{k_0}v_1^{k_1}\cdots v_{n-1}^{k_n-1}$
so that $p^{k_0+1}x=p^{k_0}v_1^{k_1+1}x=\ldots=v_{n-1}ux=0$.
To find these numbers start with $k_0$ such that $p^{k_0+1}x=0$ and $p^{k_0}x\neq 0$,
then find $k_1$ s.t. $p^{k_0}v_1^{k_1+1}x=0$, $p^{k_0}v_1^{k_1}x\neq 0$, and so on. 
One can show that this process stops using the coherence of the module, i.e. the fact 
that $\Ann(x)$ is finitely generated.
Thus, for some $n$ we have $v_n^i\notin \Ann(ux)$ for any $i\ge 0$,
and by the classification of invariant ideals (Th. \ref{th:invariant_ideals}, Cor. \ref{cr:radical})
the annihilator of $ux$ is, indeed, $I(n)$. 
The claim that $(ux)$ is invariant under Landweber-Novikov operations
follows from what we call linearity property.
Namely, if one denotes by $J_{n-1}(u)$ the ideal 
$(p^{k_0+1}, p^{k_0}v_1^{k_1+1}, \ldots, p^{k_0}v_1^{k_1}\cdots v_{n-2}^{k_{n-2}} v_{n-1}^{k_{n-1}+1})$
which is contained in the annihilator of $x$, then 
$S_{L-N}^{tot}(u)\equiv u \mod J_{n-1}(u)$ (Prop. \ref{prop:lin_landweber-novikov}).
By multiplicativity $S_{L-N}^{tot}(ux)=S_{L-N}^{tot}(u)S_{L-N}^{tot}(x)=S_{L-N}^{tot}(u)ax$
for some $a\in BP$. It follows that $S_{L-N}^{tot}(ux)=aux$,
and the submodule $BP \cdot ux$ is invariant.

However, symmetric operations do not satisfy similar linearity properties in all gradings,
which allows to use them to obtain structural results with some bounds on gradings. 
This is investigated in Section \ref{sec:linearity},
where we prove that for specific $m$ the operation $\Phi_{m}$ 
allows to divide some elements $u$ by $v_n$, 
at least modulo the same ideal $J_{n-1}(u)$ mentioned above. 
We call this {\sl non-linearity} of symmetric operations.
On the other hand, symmetric operations $\Phi_{-m}$ satisfy linearity
for sufficiently big $m$.

In Section \ref{sec:structure} we define $\Phi$-modules
as cyclic $BP$-modules with an action of the symmetric operation
and the total Landweber-Novikov operation modelling a cyclic submodule of $BP^*_{(r)}$ 
for some $r$ (Def. \ref{def:symm_module}).
The non-linearity and linearity of symmetric operations allow to construct a filtration of a $\Phi$-module
by $BP$-modules s.t. the factors are of the form $BP/I(n_i)e_i$
with $\deg e_i \ge \frac{p^{n_i}-1}{p-1}$ (Prop. \ref{prop:degree_estimates}).
More precisely, we apply the strategy explained above
and start with an element $z$ of degree $r$ in $BP^*_{(r)}$.
Then we can find a monomial $u\in \Zp[v_1,\ldots,v_{n-1}]$ s.t. $J_{n-1}(u)\subset \mathrm{Ann \ }(z)$
and therefore $I(n)uz=0$.
If the non-linearity applies (which is the condition that $\deg(uz)$ is smaller than $\frac{p^n-1}{p-1}$),
then $v_{n}^i (uz)\neq 0$, and it follows that $\Ann(uz)=I(n)$.
The linearity allows to show that the module generated by $uz$ is invariant by symmetric operations,
and thus one could consider the action of symmetric operations on $BP\cdot z/BP\cdot uz$.
It is always the case that $uz$ has positive degree, and therefore the factor-module
has less $\Zp$-generators in positive degrees. This allows the process to stop after finitely many steps.

In Section \ref{sec:filtration_integral} we glue the results on $BP^*_{(r)}$ for different $p$
to obtain a structural restriction on $\Omega^*_{(r)}$. 

\subsection{Linearity of Landweber-Novikov operations over a point}

For $n\in \NN$, $k_i\ge 0, i\colon 0\le i\le n$ and $u=p^{k_0}v_1^{k_1}\cdots v_{n}^{k_n}\in BP$ 
denote by $J_n(u)$ the ideal 
$(p^{k_0+1}, p^{k_0}v_1^{k_1+1}, \ldots, p^{k_0}v_1^{k_1}\cdots v_{n-1}^{k_{n-1}} v_{n}^{k_{n}+1})$.
For example, $J_n(1)=(p,v_1,\ldots, v_n) = I(n+1)$.

\begin{Lm}\label{lm:S_L-N(v_n)}
The restriction of the total Landweber-Novikov to $BP^*$ acts on the coefficients as follows: 
$$S_{L-N}^{tot}(v_n)\equiv v_n \mod I(n).$$

In particular, $S_{L-N}^{tot}(v_{n}^{k}) \equiv v_{n}^{k} \mod J_n(v_{n}^{k})$ for any $k\ge 1$.
\end{Lm}
\begin{proof}
The first part can be shown analogous to the proof of Lemma~\ref{lm:St_vn},
or by noting that $S_{L-N}(v_n)\in \oplus_{i=0}^{1-p^n} BP^i$,
$S_{L-N}(v_n)\in I(n+1)$ since $I(n+1)$ is an invariant ideal, 
and elements of $I(n+1)$ of degree greater than $1-p^n$ 
are contained in $I(n)$.

For the second part note that 
$J_n(v_{n}^{k})=(p,v_1,v_2,\ldots, v_{n-1}, v_n^{k+1})$ contains $I(n)$.

\end{proof}

\begin{Prop}\label{prop:lin_landweber-novikov}
Let $u=p^{k_0}v_1^{k_1}\cdots v_n^{k_n}\in BP$, then 
 $S_{L-N}^{tot}(u)\equiv u \mod J_n(u)$.
\end{Prop}
\begin{proof}

By induction on $n$. For $n=0$ we have $S_{L-N}^{tot}(u)=u$ by the additivity of $S_{L-N}^{tot}$.

By the induction assumption we have $S_{L-N}^{tot}(u)\equiv u \mod J_n(u)$ for $u=p^{k_0}v_1^{k_1}\cdots v_n^{k_n}$,
and by Lemma \ref{lm:S_L-N(v_n)} 
we have $S_{L-N}^{tot}(v_{n+1}^{k_{n+1}}) = v_{n+1}^{k_{n+1}} \mod J_{n+1}(v_{n+1}^{k_{n+1}})$.
As the operation $S_{L-N}^{tot}$ is multiplicative we have
$S_{L-N}^{tot}(uv_{n+1}^{k_{n+1}})= S_{L-N}^{tot}(u) S_{L-N}^{tot}(v_{n+1}^{k_{n+1}})$,
and, thus,  for some $x\in J_n(u)$, $y\in J_{n+1}(v_{n+1}^{k_{n+1}})$ the following holds
$$S_{L-N}^{tot}(uv_{n+1}^{k_{n+1}}) = (u+x)(v_{n+1}^{k_{n+1}}+y)
\equiv uv_{n+1}^{k_{n+1}} \mod J_{n+1}(uv_{n+1}^{k_{n+1}}),  $$
the last equality being true since ideals $J_n(u)$ and $uJ_{n+1}(v_{n+1}^{k_{n+1}})$ are contained in 
$J_{n+1}(uv_{n+1}^{k_{n+1}})=(J_n(u),uv_{n+1}^{k_{n+1}+1})$.
\end{proof}

\subsection{Linearity and non-linearity of symmetric operations over a point}\label{sec:linearity}

For any $\lambda\in BP$ the equation~(\ref{eq:symm}) defining the symmetric operation can be rewritten as

\begin{equation}\label{eq:symm_pt}
\lambda^p - St(\lambda) =^{t^{\le 0}} [p]\cdot \Phi(\lambda), \qquad \mbox{where \ } [p]=\frac{p\cdot_\Omega t}{t}.
\end{equation}

The following lemma will be useful in dealing with this equation.
If $f\in R[[t]][t^{-1}]$ (or $R[[t]]$, or $R[t^{-1}]$),
 denote by $a_k(f)$ the coefficient of monomial $t^k$ in $f$,
and denote by $\deg f=\min \{k:a_k(f)\neq 0\}$ the minimal power of $t$ 
appearing in $f$ with a non-zero coefficient.

\begin{Lm}\label{lm:solve_eq}
Let $R$ be an integral domain,
 let $f(t)\in R[[t]][t^{-1}]$, $\mathfrak{p}(t)\in R[[t]]$, $\phi(t)\in R[t^{-1}]$.
Assume that the following relation between $f,\mf{p},\phi$ holds:

\begin{equation}\label{eq:solve_eq}
 f(t) =^{t^{\le 0}} \mf{p}(t)\phi(t).
\end{equation}
Assume that $\deg f=k\le 0$, and let $\deg \mf{p}=m$.

Then $\deg \phi=k-m$, and $a_k(f)=a_m(\mf{p})a_{k-m}(\phi)$. 
\end{Lm}
\begin{proof}
As $R$ has no zero-divisors $\deg (\mf{p}\phi)=\deg \phi+\deg \mf{p}$,
and it has to equal $\deg f$ (even though the equality (\ref{eq:solve_eq})
is only true in non-positive degrees of $t$, we have $\deg f \le 0$).
Thus, comparing the coefficients of the minimal power of $t$ appearing
in the equation (\ref{eq:solve_eq}) we get
$a_{k}(\mf{p}\phi)=a_{m}(\mf{p})a_{k-m}(\mf{\phi})=a_k(f)$. 
\end{proof}

\begin{Lm}\label{lm:vn_lambda}
Let $k\ge 0$. 
\begin{enumerate} 
\item Let $i\ge 1$, then the following identity holds: \nopagebreak[4]
$$ \Phi_{(p-1)\deg(v_n^{k}v_{n+1}^i)-(p^{n+1}-1)}(v_n^{k} v_{n+1}^i)
 \equiv -v_n^{k} v_{n+1}^{i-1} \mod J_n(v_n^{k}).$$
\item Let $\lambda\in BP$, then for every $m<(p-1)\deg (v_n^{k})-(p^n-1)$ we have 
$$\Phi_{m}(v_n^{k} \lambda)\equiv v_n^{k} \Phi_{m-(p-1)\deg(v_n^{k})} (\lambda) 
\mod J_n(v_n^{k}).$$
\end{enumerate}
\end{Lm}
\begin{proof}
Note that $J_n(v_n^k)=(I(n), v_n^{k+1})$.

Assume that $k\ge 1$. For specific values the equation (\ref{eq:symm_pt}) takes form
 $v_{n}^{kp}\lambda^p-St(v_{n})^k St(\lambda)=^{t^{\le 0}} [p]\Phi(v_{n}^k\lambda)$.
By Lemma~\ref{lm:St_vn} we have $St(v_n)\equiv v_n t^{-(p-1)(p^n-1)} \mod I(n)$.
Thus, the equation above can be rewritten as 
\begin{equation*}\tag{$\star$}
v_{n}^{kp}\lambda^p-v_n^k t^{-k(p-1)(p^n-1)}St(\lambda)\equiv^{t^{\le 0}} [p]\Phi(v_{n}^k\lambda) \mod I(n).
\end{equation*}

Multiply the equation  $\lambda^p-St(\lambda)=^{t^{\le 0}} [p]\Phi(\lambda)$
by $-v_n^k$,  and add to ($\star$) multiplied by $t^{k(p-1)(p^n-1)}$ to get 

\begin{equation*}
-v_n^k \lambda^p \equiv^{t^{\le 0}} [p](\Phi(v_n^k\lambda)t^{k(p-1)(p^n-1)}-v_n^k \Phi(\lambda)) \mod I(n),
\end{equation*}

where we have ignored a monomial of positive degree in $t$ in the left hand side.

It is easy to see that $\deg ([p] \mod I(n)) = p^n-1$,
and $a_{p^n-1}([p] \mod I(n))=v_n$.
We can now apply Lemma \ref{lm:solve_eq} as $I(n)$ is a prime ideal
and get that modulo $I(n)$ for $m<-k(p-1)(p^n-1)-(p^n-1)$ 
we have $\Phi_{m}(v_n^k\lambda)-v_n^k \Phi_{m+k(p-1)(p^n-1)}(\lambda)\equiv 0$,
and that $v_n\Phi_{-(p^n-1)-k(p-1)(p^n-1)}(v_n^k\lambda)-v_n^{k+1} \Phi_{-p^n-1}(\lambda)\equiv -v_n^k\lambda^p$.
The second part of the Lemma now follows as $I(n)\subset J_n(v_n^k)$, and if $k=0$ then the claim is trivial.

Also take $\lambda = 1$ in the latter equation to get that for $k\ge 1$

\begin{equation}\label{eq:tmp}
\Phi_{-k(p-1)(p^n-1)-(p^n-1)}(v_n^k)\equiv -v_n^{k-1} \mod I(n),
\end{equation}

as $\Phi(1)=0$ and $I(n)$ is a prime ideal. 
Now to obtain the first part of the Lemma (for $k\ge 0$)
notice that $(p-1)\deg(v_n^{k}v_{n+1}^i)-(p^{n+1}-1)<(p-1)\deg (v_n^{k})-(p^n-1)$,
and using the second part we get that
$$\Phi_{(p-1)\deg(v_n^{k}v_{n+1}^i)-(p^{n+1}-1)}(v_n^{k} v_{n+1}^i)
 \equiv v_n^{k} \Phi_{(p-1)\deg(v_{n+1}^i)-(p^{n+1}-1)}(v_{n+1}^i) \mod J_n(v_n^k).$$

It suffices to prove that 
$\Phi_{(p-1)\deg(v_{n+1}^i)-(p^{n+1}-1)}(v_{n+1}^i) \equiv -v_{n+1}^{i-1} \mod I(n+1)$,
as $v_n^kI(n+1)\subset J_n(v_n^k)$.
However, this is exactly the equation (\ref{eq:tmp})
 with $n$ substituted by $n+1$, and $k$ substituted by $i$.
\end{proof}

\begin{Prop}\label{prop:symm_linearity}
Let $n\ge 1$, Let $k_j\ge 0$ for $j:0\le j \le n-1$, $m\le 0$.

Denote by $u=\prod_{j=0}^{n-1} v_j^{k_j}$.

\begin{enumerate}
\item\label{item:divide} 
For any $i\ge 1$
 $$\Phi_{(p-1)\deg(uv_n^i)-(p^n-1)}(uv_n^i)\equiv - uv_n^{i-1} \mod J_{n-1}(u).$$
 
\item\label{item:linearity} 
If $m<(p-1)\deg(u)-(p^{n-1}-1)$, 
then for any $\lambda\in BP$ 
we have 
$$\Phi_{m}(u\lambda)\equiv u\Phi_{m-(p-1)\deg(u)}(\lambda) 
\mod J_{n-1}(u).$$
\end{enumerate}
\end{Prop}
\begin{proof}

Let us prove the statement by induction on $n$.

{\bf Base of induction.} If $n=1$, then $u=p^{k_0}$, $J_0(u)=(p^{k_0+1})$,
and the second claim is obvious as $\Phi_{<0}$ is additive.
The first claim in this case is that 

$$ \Phi_{(p-1)\deg(v_1^i)-(p-1)}(p^{k_0}v_1^i)\equiv - p^{k_0} v_1^{i-1} \mod (p^{k_0+1}).$$

This follows from additivity of $\Phi_{<0}$ and the equation (\ref{eq:tmp}).

{\bf Induction step.}
Assume that for $u$ as in the assumptions of the proposition both claims (\ref{item:divide}), (\ref{item:linearity})
are true.
 Let $\tilde{u}=u v_n^{k_n}$, then $J_n(\tilde{u})=(J_{n-1}(u), uv_n^{k_n+1})$.

By induction assumption we know 
that $\Phi_{m}(u v_n^{k_n} \lambda)\equiv u\Phi_{m-(p-1)\deg(u)}(v_n^{k_n} \lambda) \mod J_{n-1}(u)$
for every $\lambda \in BP$ and every $m<(p-1)\deg(u)-(p^{n-1}-1)$.
Clearly, for any $i\ge 1$ we have $(p-1)\deg(uv_n^{k_n}v_{n+1}^i)-(p^{n+1}-1)< (p-1)\deg(u)-(p^{n-1}-1)$,
and therefore 
$$\Phi_{(p-1)\deg(\tilde{u}v_{n+1}^i)-(p^{n+1}-1)}(\tilde{u}v_{n+1}^i)\equiv 
u\Phi_{(p-1)\deg(v_n^{k_n}v_{n+1}^i)-(p^{n+1}-1)}(v_n^{k_n} v_{n+1}^i) \mod J_{n-1}(u).$$ 

However, by Lemma \ref{lm:vn_lambda}, (1) we have
 $\Phi_{(p-1)\deg(v_n^{k_n}v_{n+1}^i)-(p^{n+1}-1)}(v_n^{k_n} v_{n+1}^i) \equiv -v_n^{k_n} v_{n+1}^{i-1} \mod 
J_n(v_n^{k_n})$,
and the claim (\ref{item:divide}) is proved, as $u J_n(v_n^{k_n}) \subset J_n(\tilde{u})$
and $J_{n-1}(u)\subset J_n(\tilde{u})$.

To prove the claim (\ref{item:linearity})
it suffices to show that 
for  $m<(p-1)\deg(\tilde{u})-(p^n-1)$ we have 
$$\Phi_{m-(p-1)\deg(u)}(v_n^{k_n} \lambda)\equiv v_n^{k_n} 
\Phi_{m-(p-1)\deg(u)-(p-1)\deg(v_n^{k_n})} (\lambda) 
\mod J_n(v_n^{k_n}).$$

However, this is precisely the statement of Lemma \ref{lm:vn_lambda}, (2).
\end{proof}

\subsection{Structure of modules with the action of symmetric operations}\label{sec:structure}

In this section we give an algebraic description 
of a cyclic submodule of $BP^*_{(r)}$ with the action
of symmetric and Landweber-Novikov operations.
The results of previous section allow us to prove severe restrictions
on such modules. Later we apply these results to the structure of $BP^*_{(r)}$.

Recall that there is a surjective $BP$-linear map $\rho_{BP}:\CH^r(X)\ot BP \rarr BP^*_{(r)}$ 
for any smooth quasi-projective variety $X$.
Let $z$ be an element of $\CH^r(X)$
and denote by $x$ its image $\rho_{BP}(z)$ in $BP^r_{(r)}(X)$.
By Prop.~\ref{prop:LN_action_cobord_top},~\ref{prop:symm_action_bp}
 symmetric operations and Landweber-Novikov operations 
act on $x$ by formulas: 
$$ \Phi(\lambda x) =  \mathbf{i}^r\cdot t^{r(p-1)}\cdot \Phi_{\le -r(p-1)}(\lambda) x,
\qquad S^{tot}_{L-N}(\lambda x) = S^{tot}_{L-N}(\lambda) x. $$
Note that this action is linear in $x$, and the coefficient depends only on $r$ and not on $x$ or $X$.
Recall also that $\Phi_{<0}$ is additive so in the formula above $\Phi$ acts additively for $r>0$.

\begin{Def}\label{def:symm_module}
Let $M$ be a graded cyclic coherent $BP$-module generated by an element $x$ in degree $r>0$.
Denote by $I$ the annihilator of $x$ which is a finitely generated ideal in $BP$.
Suppose we are given two additive maps $\Phi_M:BP/I\rarr BP/I[t^{-1}]$ and $S_M:BP/I\rarr BP/I[b_1,b_2,\ldots]$.

We call $(M, \Phi_M, S_M)$ {\sl a  $\Phi$-module of degree $r$}
if for any $\lambda\in BP$  we have
$$ \Phi_M(\lambda) \equiv \mathbf{i}^r\cdot t^{r(p-1)}\cdot \Phi_{\le -r(p-1)}(\lambda) \mod I,
\qquad S_M(\lambda) \equiv S^{tot}_{L-N}(\lambda) \mod I.$$
\end{Def}

In other words, $\Phi_M$ and $S_M$ model an action of symmetric and Landweber-Novikov operations on $M$
and we will use notation $\Phi(\lambda x):=\Phi_M(\lambda)x$ and $S_{L-N}^{tot}(\lambda x)= S_M(\lambda) x$.
We will also often denote a $\Phi$-module $(M, \Phi_M, S_M)$ by $M$
since maps $\Phi_M, S_M$ are uniquely defined by the degree of the generator of $M$.

This definition is a rough approximation to a definition of a module over the
algebra of operations in $BP^*_{(r)}$ generated by symmetric and Landweber-Novikov operations.
This algebra has a complicated structure since one has to work with operations between different degrees 'separately'.
However, as we will not need to investigate relations between compositions of operations
we prefer an explicit definition~\ref{def:symm_module}. 
 
From now on we will use the collective term {\sl operations}
to denote symmetric and Landweber-Novikov operations. 
 
\begin{Lm}\label{lm:factor-phi-module}
Let  $(M, \Phi_M, S_M)$ be a $\Phi$-module of degree $r$, and assume that 
the submodule generated by an element $ux=:y\in M$
is invariant under operations. 

Denote by $I:=(x:BP\cdot y)=\{\lambda \in BP| \lambda x\in BP\cdot y\}$
the annihilator ideal of the generator 
of the module $M/BP\cdot y$.
Then $(M/BP\cdot y, \Phi_M \mathrm{\ mod\ } I, S_M \mathrm{\ mod\ } I)$ is a $\Phi$-module of degree $r$.
\end{Lm} 
\begin{proof}
The module $M/BP\cdot y$ is coherent since coherent modules form an abelian category.
We need to check that maps $\Phi_M$, $S_M$ factor through $I$.
Let $\mu\in I$, then $\Phi_M(\mu) \in (I \mod \Ann_M x)$ since $\Phi(\mu x)\in BP[t^{-1}] \cdot y= I[t^{-1}]\cdot x$.
 Similarly, for $S_M$.
\end{proof}
  
In general, for a smooth variety $X$ the $BP$-module $BP^*_{(r)}(X)$ can be infinitely generated,
but its finitely generated submodules are actually
glued from $\Phi$-modules of degree $r$.

\begin{Prop}\label{prop:fin-gen-BP-is-ext-symm-mod}
Let $r>0$ and let $Z\subset \CH^r(X)$ be a finitely generated group,
then the $BP$-module $M$ generated by $\rho(Z)$ in $BP^*_{(r)}(X)$ (the map $\rho$ is defined in Section \ref{sec:top_filt}) 
has a finite filtration s.t. its factors are $\Phi$-modules of degree $r$. 
\end{Prop}
\begin{proof}
Note that if $Z'\subset Z$ is a subgroup, then its image inside $BP^*_{(r)}(X)$ 
is invariant with respect to  operations.

Being finitely generated the group $Z$ has a finite filtration,
s.t. its factors are cyclic. This implies,
that $M$ has a finite filtration, s.t. its factors are cyclic modules generated in degree $r$
with well-defined action of operations.
Clearly, this action satisfies Def. \ref{def:symm_module}.
The coherence follows from Prop. \ref{prop:finitely_presented}.
\end{proof}

Denote by $f(n)=\frac{p^n-1}{p-1}$ for $n>0$.

\begin{Prop}[cf. {\cite[Cor. 21.9]{BJW}}]\label{prop:degree_estimates}
Let $M$ be a $\Phi$-module of degree $r>0$ generated by $x$.

Then either $M\cong BP$, 
or there exist $n>0$ and a $BP$-submodule $BP\cdot y \subset M$ which is invariant
under operations, $\Ann(y) = I(n)$ and $\deg y \ge f(n)$.
\end{Prop}

\begin{proof}
Let us construct  $u=p^{k_0}v_1^{k_1}\cdots v_{n-1}^{k_{n-1}}\in BP$ so that 
$\Ann (x)\supset J_{n-1}(u)$, 
$y:=ux\neq 0$, and $\deg y \ge f(n)$.

Note that $\Ann(x)$ is an invariant finitely generated ideal in $BP$,
and if $p^{k_0+1}x\neq 0$ for every $k_0\ge 0$, then by Cor. \ref{cr:radical} we have
$\Ann(x)=0$, and so $M\cong BP$ is a free module generated in degree $r>0$.
Thus, we may assume that $p^{k_0+1}x=0$ and take $y_0:=p^{k_0}x\neq 0$ for some $k_0\ge 0$.

We will continue by induction. 
Define $y_n$ as $v_n^{k_n} y_{n-1}$ such that $v_n y_n=0$,
and $y_n\neq 0$. If $v_n^{i} y_{n-1}\neq 0$ for any $i\ge 0$, then
the process stops and $y=y_{n-1}$.
Note that by construction $\Ann (x)\supset J_n(u_n)$ 
where $y_n= u_n x$.
The following lemma shows that this induction process terminates. 

\begin{Lm}\label{lm:v_n-annihilator-degree}
Let $n>0$ and assume that $u=p^{k_0}v_1^{k_1}\cdots v_{n-1}^{k_{n-1}}\in BP$ such that 
$\Ann (x)\supset J_{n-1}(u)$, 
$y:=ux\neq 0$ and $\deg y \le pf(n)$.

Then $v_{n}^iy \neq 0$ for all $i\ge 0$.
\end{Lm}
\begin{proof}
Assume the contrary, and choose $i\ge 1$ to be 
such that $v_n^{i-1}y\neq 0$, $v_n^i y=0$.

Let us apply the symmetric operation to $v_n^i y$, which would have to result in zero.
By the definition of a $\Phi$-module 
we have $\Phi(v_n^iy)=\Phi(uv_n^i x)= \mathbf{i}^r\cdot t^{r(p-1)}\cdot 
  \Phi_{\le -r(p-1)}(uv_n^i)\cdot x$.
By Prop. \ref{prop:symm_linearity}, (\ref{item:divide})
$\Phi_{(p-1)\deg(uv_n^i)-(p^n-1)}(u v_n^i) \equiv -uv_n^{i-1} \mod J_{n-1}(u)$.
Thus, if $-r(p-1)\ge (p-1)\deg (uv_n^i)-(p^n-1)$,
then $\Phi_{(p-1)(r+\deg(uv_n^i))-(p^n-1)}(v_n^iux) = \mathbf{i}^r v_n^{i-1}ux=\mathbf{i}^r v_n^{i-1}y$ as $J_{n-1}(u)x=0$.
If $v_n^iux=v_n^iy$ were 0, then $\Phi(v_n^iux)$ would be 0 which is a contradiction.

However, the inequality above can be rewritten 
as $(p-1)(r+\deg(uv_n^i))\le p^n-1$,
and it is satisfied for all $i\ge 1$
if it is satisfied for $i=1$. In this case this inequality is precisely
the assumption of the Lemma.
\end{proof}

Denote by $y=ux$ the final element of the process above, so that $v_n^i y\neq 0$ for all $i$.
Let us show that the $BP$-submodule generated by $y$ 
is stable under operations.
We have $\Phi(\lambda y)=x\cdot \mathbf{i}^r\cdot t^{r(p-1)}\cdot \Phi_{\le-r(p-1)}(u\lambda)$ 
for any $\lambda\in BP$.
To apply Prop. \ref{prop:symm_linearity}, \ref{item:linearity})
 we need that $-r(p-1)< (p-1)\deg (v_1^{k_1}\cdots v_{n-1}^{k_{n-1}})-(p^{n-1}-1)$.
This is $(p-1)\deg y> p^{n-1}-1$, or $\deg y>f(n-1)$.

As we have $v_{n-1}y=0$, then it follows from the proof of Lemma~\ref{lm:v_n-annihilator-degree}
that $\deg y >pf(n-1)$, i.e. $\deg y\ge f(n)=pf(n-1)+1$.
 Clearly, $f(n)-1>f(n-1)$, and therefore for $m\le -r(p-1)$ we have $\Phi_{m}(u\lambda)$ 
 equal to $u \Phi_{m-(p-1)\deg u}(\lambda)$ modulo the ideal $J_{n-1}(u)$ 
 by Prop. \ref{prop:symm_linearity}, \ref{item:linearity}).
 However, $J_{n-1}(u) \subset \Ann(x)$,
 and therefore $\Phi_l$ sends $\lambda y$ to $\lambda' y$ 
 for any $l\le 0$.

We want to show now that the submodule generated by $y$ is also invariant
 with respect to Landweber-Novikov operations.
We have $S_{L-N}^{tot}(\lambda y)=S_{L-N}^{tot}(\lambda)S_{L-N}^{tot}(u)x$.
However, $J_{n-1}(u)\subset \Ann(x)$ and by Prop.~\ref{prop:lin_landweber-novikov}
$S_{L-N}^{tot}(u)\equiv u \mod J_{n-1}(u)$, so that $S_{L-N}^{tot}(\lambda y)=S_{L-N}^{tot}(\lambda)y$.
This finishes the proof that a submodule generated by $y$ is invariant under operations.

It also follows that $\Ann(y)$ is an invariant ideal, s.t. $\Ann(y)\supset I(n)$, $v_{n}^i\notin \Ann(y)$ for all $i\ge 0$.
It is finitely generated, as $\Phi$-modules are coherent,
and by Cor. \ref{cr:radical} we obtain that $\Ann(y)=I(n)$.

\end{proof}

\begin{Prop}[cf. {\cite[Th. 21.12]{BJW}}]\label{prop:abs_filtration_BP}
Let $M$ be a $\Phi$-module of degree $r>0$.

Then either $M$ is a free module 
or there exist a finite filtration of $M$: $M=M^{0}\supset M^{1}\supset \cdots M^{s}=0$
by $BP$-modules
s.t. $M^{i}/M^{i+1}$ is isomorphic to a $BP$-module $BP/I(n_i)e_i$
where $\deg e_i\ge f(n_i)$.
\end{Prop}

\begin{proof}

Applying inductively Prop. \ref{prop:degree_estimates} and Lemma \ref{lm:factor-phi-module} 
we obtain such filtration, the only issue being its finiteness (and hence exhaustiveness).
To see this note, that after each step of the induction process 
an element of positive degree is killed,
and as $M^{\ge 0}$ is a finitely generated $\Zp$-module, 
the process stops after finite number of steps due to the noetherian property of $\Zp$.
\end{proof}

\begin{Cr}\label{cr:filtration_BP}
Let $X$ be a smooth variety,
then for any $r\ge 0$ the $BP$-module
$BP^*_{(r)}(X)$ is a union of finitely generated $BP$-modules
which have a filtration s.t. its factors 
 are cyclic modules $BP/I(n) x$
with the generator $x$ s.t. $\deg x\ge f(n)$.
(If $n=0$, i.e. $I(n)=(0)$, the claim is that $\deg x\ge 0$.)
\end{Cr}
\begin{proof}
It follows from Prop. \ref{prop:fin-gen-BP-is-ext-symm-mod} and Prop. \ref{prop:abs_filtration_BP}.
\end{proof}

\subsection{Integral restrictions to the structure of cobordism}\label{sec:filtration_integral}

Let us now use Landweber's results 
on the structure of $(\LL,\LL B)$-comodules (Th. \ref{th:landweber_filtration})
to state the integral version of Cor. \ref{cr:filtration_BP}.

\begin{Th}\label{th:filtration}
Let $X$ be a smooth quasi-projective variety over a field $F$.
Then for any $r\ge 0$ the $\LL$-module $\Omega^*_{(r)}(X)$
is a union of $\LL$-modules
which have a filtration, s.t. its factors are
of the form $\LL e_j$, $\deg e_j\ge 0$,  or $\LL/I(p_i,n_i)\cdot e_i$ where $p_i$ is a prime number,
$n_i\ge 1$, and $\deg e_i \ge \frac{p^{n_i}-1}{p-1}$.
\end{Th}
\begin{proof}
For $r=0$ the module $\Omega^*_{(r)}(X)$ is free of rank 1, thus, the claim is satisfied,
and in what follows we assume that $r>0$.

Let $Z$ be a finitely generated subgroup of $\CH^r(X)$,
and denote by $M$ the graded $\LL$-module that is generated by $\rho(Z)$.
Abelian group $Z$ has a filtration $Z=Z_0\supset Z_1 \supset \cdots Z_s=0$ 
s.t. its factors are either $\ZZ$ or $\ZZ/p_j$ for a prime $p_j$,
and therefore the module $M$ has a filtration $M=M_0\supset M_1\supset \cdots \supset M_s=0$
s.t. $M_j/M_{j+1}$ is a cyclic $\LL$-module generated by the element $x_j$ 
s.t. $x_j$ is either non-torsion or $p_j x_j=0$
for a prime $p_j$.

Note that the filtration is invariant under Landweber-Novikov operations
 due to Prop.~\ref{prop:LN_action_cobord_top}, and thus 
 $M_j/M_{j+1}$ is a $(\LL,\LL B)$-comodule, which is also coherent by Prop. \ref{prop:finitely_presented}
and Rem. \ref{rem:indcoh_abelian_cat}.

Let us show that the module $N:=M_j/M_{j+1}$ has a filtration as claimed in the theorem.
If $x_j$ is non-torsion, then this module is isomorphic to $\LL$
as the annihilator of $x_j$ has to be an invariant ideal in $\LL$ and
its radical (if not equal to 0) has to contain a prime number $p$
by Th.~\ref{th:radical-invariant-ideal} which is a contradiction.
So we assume that $p x_j=0$ for some prime $p$.
Note that it makes $N$ a $\LL_{(p)}$-module.
It follows from Prop. \ref{prop:BP_Omega_algebraic} 
that $N\cong (N\ot_\LL BP)\ot_{BP} \LL_{(p)}$.

We claim that the module $N\ot_\LL BP$ is a $\Phi$-module generated in degree $r$.
Indeed, $M\ot_\LL BP$ is a submodule of $BP^*_{(r)}(X)$
because $BP$ is a Landweber-exact module over $\LL$,
and moreover it is generated by $\rho_{BP}(Z)$ as the following diagram commutes:
\begin{diagram}
\CH^r(X) &\rTo^{\rho} &\Omega^*_{(r)}(X)\\
& \rdTo_{\rho_{BP}} &\dTo \\
& & BP^*_{(r)}(X)
\end{diagram}
The induced filtration on $M\ot_\LL BP$ then 
equals to the filtration given by the $\rho_{BP}$-image of the filtration
on $Z$. This filtration is clearly invariant under all operations,
and therefore its factors being cyclic are $\Phi$-modules.
Therefore by Prop. \ref{prop:abs_filtration_BP} it has an expected filtration.

The map $BP\rarr \LL_{(p)}$ is a flat map,
so that the graded factors of the induced filtration on $N$
are of the form $BP/I(n)\ot_{BP} \LL_{(p)}$. 
As $I(n)\ot_{BP} \LL_{(p)}=I(p,n)$ this gives an expected filtration on 
$N=M_j/M_{j+1}$ 
and therefore $M$ as well.
\end{proof}

\section{Algebraic cobordism and BP-theory of varieties of small dimension}\label{sec:cobord_surface}

Vishik has used Theorem \ref{th:relations_positive} to show
that for a smooth curve $C$ there is an isomorphism $\Omega^*(C)=\LL\oplus \CH^1(C)\ot\LL$
(\cite[Th. 4.4]{Vish_Lazard}).
Here we continue this line of results with the case of algebraic cobordism of a smooth surface
and $BP$-theory of a smooth variety of dimension not greater than $p$.

To state the results we will need the following Lemma and Definition.
Let $A, B$ be abelian groups, and let $L=\oplus_{i\le 0} L^i$ be a graded ring which is flat over $\ZZ$.
For an abelian  group $X$ denote by $X\ot L(i)$
 the graded $L$-module generated freely by $X$ in degree $i$.

\begin{Lm}\label{lm:ext_specified}
$\Ext^1_{L^\bullet}(A\ot L(i), B\ot L(j))=\Ext^1_{\ZZ}(A,B\otimes L^{i-j})$.
\end{Lm}
\begin{proof}
The $L$-module $A\ot L(i)$ has a free resolution obtained from a free resolution 
of $A$ by abelian groups. Calculating the $\Ext$-groups using this resolution yields the claim.
\end{proof}

\begin{Def}\label{def:ext_specified}
Let $C$ be a graded $L$-module which is
 an extension of a module $A\ot L(i)$ by $B\ot L(j)$ where $j\le i$.
Let $v$ be an element $L^{i-j}$. 
We will say that $C$ is {\sl specified by an extension in $\Ext^1(A,B)$ via $v$},
 if the extension defined by $C$ in $\Ext^1(A,B\ot L^{i-j})$ via Lemma \ref{lm:ext_specified}) 
 comes from the image of the map $\Ext^1(A,B)\rarr \Ext^1(A,B\ot L^{i-j})$ defined by $v$.
\end{Def}

\subsection{Algebraic cobordism of a surface}

Recall that for any smooth variety $X$ we have a decomposition
of $\LL$-modules $\Omega^*=\LL\cdot 1_X\oplus \tau^1\Omega^*$,
and therefore a description of $\tau^1\Omega^*$ as a $\LL$-module
yields a description of $\Omega^*$ as a $\LL$-module.

\begin{Th}\label{th:cobord_surface}
Let $S$ be a smooth surface.
Then there exists the following exact sequence of $\LL$-modules
$$ 0\rarr \CH^2(S)\ot\LL(2)\rarr \tau^1\Omega^*(S)\rarr \CH^1(S)\ot\LL(1)\rarr 0,$$
and this extension of $\LL$-modules 
is specified by an extension of abelian groups via $\beta \in \LL^{-1}$
$$ 0\rarr \CH^2(S)\cong \tau^2 K_0(S)\rarr \tau^1 K_0(S)\rarr gr_\tau^1 K_0(S)\cong \CH^1(S)\rarr 0.$$
\end{Th}

\begin{proof} 
As explained in the proof of Prop. \ref{prop:finitely_presented},
$\Omega^*(S)$ as $\LL$-module is free with respect to variables of degree less than -1,
i.e. $\Omega^*(S)=CK_0^*(S)\ot_{\ZZ[\beta]} \LL$,
where $CK_0^*$ is the free theory defined by the multiplicative
formal group law over the ring $\ZZ[\beta]$, $\deg \beta =-1$, and the map
$\ZZ[\beta]\rarr \LL$ is the section of the classifying map
and sends $\beta$ to the generator of $\LL^{-1}$.
The theory $CK_0^*$ is called the {\sl connective K-theory}.

Note that the canonical map $CK_0^2(S)\rarr \CH^2(S)$ is an isomorphism,
as its kernel has to be divisible by $\beta$, which is impossible due to the dimensional reasons.

\begin{Lm}\label{lm:CH_CK0}
Let $Y$ be a smooth variety.
The map $\rho_{CK_0}:\CH^i(Y)\ot\ZZ[\beta] \rarr \oplus_{n\le i} CK^n_{0,(i)}(Y)$ is an isomorphism
for $i\colon 0\le i\le 2$.
\end{Lm}
\begin{proof}
If $i=0$, then there is a canonical section.

If $i=1$, the statement is a corollary of Vishik's theorem (Cor. \ref{cr:CH1_Vishik}).

If $i=2$, as the map $\rho_{CK_0}$ is always surjective,
 it suffices to prove that for any non-zero element $x\in\CH^2(Y)$ elements 
$\beta^i \rho_{CK_0}(x)$ are not zero for any $i\ge 0$.  Assume the contrary. 

We claim that $x$ is 2-primary torsion. 
Since every two formal group laws over the same $\QQ$-algebra are strictly isomorphic,
it follows that every two free theory with torsion-free coefficients become multiplicatively
isomorphic after a change of coefficient to some $\QQ$-algebra.
Using this one can show that the kernel of the map $\rho$ is torsion,
and therefore $x$ is torsion.
 Let $n\in\ZZ$ be the order of $x$, and assume that $p\mid n$, $p\neq 2$.
Then there exist a multiple of $x$, called $y$, such that $py=0$.
 Clearly, $\rho_{BP}(y)$ is a non-zero element
in $BP^2_{(2)}(Y)$ (where $BP$ corresponds to a prime $p$). Note that as the map $\LL\rarr BP$ is graded,
it sends $\beta\in \LL^{-1}$ to zero if $p\neq 2$. On the other hand, by Prop. \ref{prop:from_Omega_to_BP}
the module $\Omega^*_{(2)}(Y)\ot\Zp$ is isomorphic to a polynomial ring over $BP^*_{(2)}(Y)$
 as $\LL$-module.
Clearly, $\beta$ is mapped to a generator of this polynomial ring, 
and therefore $\beta^i y \neq 0$ for any $i\ge 1$.

If $x$ is 2-primary torsion, then its image in $BP^2_{(2)}(Y)\cong^{\rho^{-1}} \CH^2(Y)\ot\ZZ_{(2)}$
 corresponding to prime $2$ is non-zero
and is annihilated by the multiplication on a power of $v_1$ (which is the image of $\beta$). 
Thus, $x$ generates a non-zero $\Phi$-module of degree 2 inside $BP^2_{(2)}(Y)$.

Consider the filtration on this $\Phi$-module from Prop. \ref{prop:abs_filtration_BP}.
As $\deg x<f(2)=\frac{2^2-1}{2-1}=3$, the highest factor of the filtration can only be $BP/I(1)$
where $I(1)=(2)$, i.e. this $\Phi$-module is isomorphic to $BP/(2^s)$ for some $s>0$,
and $v_1^l x \neq 0$ for any $l\ge 0$. Contradiction. 
\end{proof}

It follows from this Lemma, and the isomorphism $\CH^*=CK^*_0/(\beta)$, 
that the following sequence of abelian groups
is exact:
\begin{equation}\label{eq:CK0}
 0 \rarr \CH^2(S)[\beta]\cong \tau^2 CK_0^*(S)\rarr \tau^1 CK_0^*(S) \rarr \tau^1 CK_0^*(S)/\tau^2 CK_0^*(S)\cong^{\rho^{-1}}
\CH^1(S)[\beta] \rarr 0.
\end{equation}

The first piece of the topological filtration on $\ZZ[\beta]$-module $CK^*(S)$ is split by a 
free $\ZZ[\beta]$-summand generated by 1.
In other words, $CK_0^*(S)=\ZZ[\beta]\oplus \tau^1CK_0^*(S)$.

The exact sequence (\ref{eq:CK0}) may be not split,
but the extension it defines as of $\ZZ[\beta]$-modules
is specified by an extension of abelian groups $\CH^1(S)$ and $\CH^2(S)$ via element $\beta$,
see Lemma~\ref{lm:ext_specified},
which can be obtained by localizing at $\beta=1$.

Tensoring the exact sequence (\ref{eq:CK0}) over $\ZZ[\beta]$ by $\ZZ[\beta]/(\beta -1)$
and noting that $\Tor_1^{\ZZ[\beta]}$ between $\CH^1(S)[\beta]$ and $\ZZ[\beta]/(\beta -1)$
is equal to zero
we obtain the following:

$$ 0\rarr \CH^2(S)\rarr \tau^1 K_0(S) \rarr \CH^1(S) \rarr 0.$$

It is easy to check that the maps in this sequence
are $\rho_{K_0}$ which identify $\CH^i(S)$ with $gr^i_\tau K_0(S)$ for $i=1, 2$.
This finishes the proof of the theorem.
\end{proof}

\subsection{BP-theory of varieties of small dimension}

Recall that the structure of algebraic cobordism $p$-locally can be
recovered from the structure of $BP^*$ by Prop. \ref{prop:from_Omega_to_BP}.
Thus, one can reformulate following results in terms of $\Omega^*\ot\Zp$ instead of $BP^*$.

\begin{Lm}\label{lm:CH_BP_top_dim_p}
Let $Y$ be a smooth variety.
The map $\rho_{BP}:\CH^i(Y)\ot BP \rarr \oplus_{n\le i} BP^n_{(i)}(Y)$
is an isomorphism for $i\colon 0\le i \le p$.
\end{Lm}
\begin{proof}
By Cor. \ref{cr:filtration_BP} we know that the $BP$-module $BP^*_{(i)}(Y)$
has a filtration whose factors are either $BP$ or $BP/I(1)=BP/(p)$ for $i<f(2)=p+1$
generated in degree $i$ (note that in the construction of the 
filtration in Prop.~\ref{prop:degree_estimates}
the generator $y$ of a submodule $BP/I(n)$ was obtained 
as a product of $x$ with elements of $I(n)$).
Since $\rho_{BP}$ is injective in degree $i$ and surjective in all degrees,
it follows that it is an isomorphism for $i\le p$.
\end{proof}

Similar to the case of algebraic surface we can show 
that if $\dim X\le p$ then the $BP$-module $BP^*(X)$ is equal to $CK(1)^*(X)\ot_{\Zp[v_1]} BP$
where $CK(1)^*$ is a first Morava K-theory which we now define.

\begin{Prop}\label{prop:artin-hasse}
Let $F$ be a formal group law over $\QQ$
defined by the logarithm $\log_{K(1)}(x)=\sum_{i=0}\frac{x^{p^i}}{p^i}$.

Then $F$ is strictly isomorphic to a multiplicative formal group law,
and this isomorphism has coefficients in $\Zp$. 
Thus, the coefficients of $F$ also lie in $\Zp$.
\end{Prop}
\begin{proof}
The strict isomorphism of formal group laws $F_1, F_2$ over $\QQ$ is unique,
and is defined by the series $\log_{F_1}\circ \log_{F_2}^{-1}$.
Thus, one needs to check that $\exp\circ \log_{K(1)}$ has coefficients in $\Zp$.
This series is known as the Artin-Hasse exponential 
and is known to be $p$-integral.
\end{proof}

If a formal group law $F(x,y)=\sum a_{ij} x^iy^j$, $a_{ij}\in \Zp$ has
$a_{ij} =0$ whenever $i+j\neq 1 \mod p-1$,
then the formal group law $F(x,y)=\sum a_{ij} v_1^{\frac{i+j-1}{p-1}} x^iy^j$
is a graded formal group law over $\Zp[v_1]$, $\deg v_1=1-p$.
Moreover, if the logarithm of $F$ is $\sum a_i x^{p^i}$,
then the logarithm of the graded version is $\sum a_i v_1^{\frac{p^i-1}{p-1}} x^{p^i}$.
This proves the correctness of the following definition.

\begin{Def}
Let $F_{CK(1)}$ be a formal group law over the graded ring $\Zp[v_1]$, $\deg v_1=1-p$,
defined by the logarithm $\log_{CK(1)}(x)=\sum_{i=0}v_1^\frac{p^i-1}{p-1}\frac{x^{p^i}}{p^i}$.
The corresponding free theory $CK(1)^*$ is called {\sl the first connective Morava K-theory}.

The free $\ZZ/(p-1)$-graded theory $K(1)^*=CK(1)^*/(v_1-1)$
is called {\sl the first Morava K-theory}.
\end{Def}

We will denote graded components of $K(1)^*$ 
by $K(1)^i$ where $i\in \ZZ$ meaning that $K(1)^i=K(1)^j$ if $i\equiv~j~\mod~(p-1)$.
Thus, $K(1)^*=K(1)^1 \oplus K(1)^2 \oplus \cdots \oplus K(1)^{p-1}$.

\begin{Lm}\label{lm:CK(1)_top}
Let $Y$ be a smooth variety.

Then $\tau^i CK(1)^j = \tau^{i+1} CK(1)^j$ , $\tau^i K(1)^j = \tau^{i+1} K(1)^j$
if $j\neq i \mod (p-1)$.

In particular, if $\dim Y\le p$,
then $\oplus_n CK(1)^n_{(i)}(Y)=\oplus_{m=0}^\infty CK(1)^{i-m(p-1)}(Y)$
 and $K(1)^i(Y)=K(1)^*_{(i)}(Y)$ 
for $i\colon 2\le i \le p-2$, $K(1)^{p-1}(Y)=K(1)\oplus K(1)^*_{(p-1)}(Y)$.
\end{Lm}
\begin{proof}
For a $\ZZ$-graded free theory $A^*$ such that the map $\LL\rarr A$ is surjective
we have $\tau^i A^n(X) = \cup_{m\le -i} A^{m}\cdot  A^{n-m}(X)$ for every smooth variety $X$
as follows from \cite[Th. 4.5.7]{LevMor}.
However, the ring of coefficients of $CK(1)^*$ has non-zero elements 
only in degrees dividing $p-1$. 
The result for $K(1)^*$ follows by localization
and using the canonical splitting $K(1)^*=K(1)\oplus \tau^1 K(1)^*$.
\end{proof}

\begin{Th}\label{th:BP_pfold}
Let $p$ be a prime number, let $X$ be a smooth variety, $\dim X\le p$.

Then the graded $BP$-module $BP^*(X)$ 
is a direct sum of $BP\cdot 1_X$, $\CH^j(X)\ot BP(j)$ for $j\colon 2\le j \le p-1$
and a graded $BP$-module $M$ which 
fits into the following exact sequence
$$ 0\rarr \CH^p(X)\ot BP(p) \rarr M\rarr \CH^1(X)\ot BP(1)\rarr 0,$$
and this extension of $BP$-modules 
is specified by the following extension of abelian groups via the element $v_1\in BP^{1-p}$ 
$$ 0\rarr \CH^p(X)\ot\Zp \cong \tau^p K(1)^1(X)\rarr \tau^1 K(1)^p(X)
\rarr gr_\tau^1 K(1)^p(X)\cong \CH^1(X)\ot\Zp\rarr 0.$$
\end{Th}

\begin{proof}

As follows from the proof of Prop. \ref{prop:finitely_presented} 
we have $BP^*(X)= CK(1)^*(X)\ot_{\Zp[v_1]} BP$. 
Thus, we need to focus only on the structure of the first connective Morava K-theory.

It follows from Lemmata \ref{lm:CH_BP_top_dim_p} and \ref{lm:CK(1)_top}
that $\oplus_{n\ge 0} CK(1)^{i-n(p-1)}=\CH^i\ot\Zp[v_1]$ for $i\colon 2\le i \le p-1$,
and we have the following exact sequence

$$ 0\rarr \tau^p CK(1)^{p+*(p-1)}(X) \rarr \tau^1 CK(1)^{p+*(p-1)}_{(p)}(X)
\rarr CK(1)_{(1)}^*(X)\rarr 0.$$ 

Note that $CK(1)^{p+*(p-1)}(X)\cong \CH^p(X)\ot\Zp[v_1](p)$,
and $CK(1)_{(1)}^*(X)\cong \CH^1(X)\ot\Zp[v_1](1)$,
so the exact sequence is specified by an extension of Chow groups
via element $v_1$ (see Lemma~\ref{lm:ext_specified}).
Tensoring this exact sequence over $\Zp[v_1]$ by $\Zp[v_1]/(v_1 -1)$
we obtain the exact sequence

$$ 0\rarr \CH^p(X)\ot\Zp\cong \tau^p K(1)^*(X) \rarr \tau^1 K(1)^1(X) \rarr K(1)^*_{(1)}\cong \CH^1(X)\ot\Zp\rarr 0.$$

\end{proof}

There is a simple connection between the first Morava K-theory
and $K_0$ which we now state and which can be used to restate Theorem \ref{th:BP_pfold}
in more classical terms.

\begin{Prop}\label{prop:K0-p-local}
Let $X$ be a smooth variety.
Then there exist a multiplicative isomorphism 
$$K_0(X)\ot\Zp \cong \oplus_{j=1}^{p-1} K(1)^j(X).$$

In particular, if $\dim X\le p$,
 this isomorphism identifies 
\begin{itemize}
\item  $\tau^i K_0(X)\ot\Zp$ with 
 {$\tau^p \tilde{K}(1)^1(X) \bigoplus \oplus_{j=i}^{p-1} \tilde{K}(1)^j(X)$}
for $i\colon 2\le i\le p-1$;
\item $\tau^p K_0(X)\ot\Zp$ with $\tau^p \tilde{K}(1)^1(X)$
\end{itemize} 
where $\tilde{K}(1)^* = \tau^1 K(1)^* = \mathrm{Ker\ }(K(1)^*\rarr \Zp)$ is the subideal
of elements which vanish at generic points.
\end{Prop}
\begin{proof}
The existence of a multiplicative isomorphism follows by Prop. \ref{prop:artin-hasse}
and the Panin-Smirnov reorientation construction.
Clearly, this isomorphism respects the topological filtration,
and the claim then follows from Lemma \ref{lm:CK(1)_top}.
\end{proof}

We also note that the projector of $K_0\ot\Zp$ 
which image is isomorphic to $K(1)^i$ where $i\colon 1\le i\le p-1$
can be written down with the help of Chern classes $c_j:K_0\rarr K_0$ 
or in terms of $\lambda$-operations on $K_0$.

\section{Resolutions of algebraic cobordism over the Lazard ring}\label{sec:resolutions}

In this section we prove the Syzygies conjecture:

\begin{Conj}[Vishik, {\cite[Conj. 4.5 p. 981]{Vish_Lazard}}]
Let $X$ be a smooth variety of dimension $d$. 
Then $\Omega^*(X)$ has a free $\LL$-resolution whose $j$-th term has generators concentrated 
in codimensions between $j$ and $d$. 

In particular, the cohomological $\mathbb{L}$-dimension
of the $\mathbb{L}$-module is less or equal to $d$.
\end{Conj}

First, we describe a homological criterion for a $\LL$-module to have a free resolution as above,
and then show that the results of Th. \ref{th:filtration} imply that these
conditions are satisfied for the algebraic cobordism.

Recall that $\LL/\LL^{<0}=\ZZ$.

\begin{Prop}\label{prop:tor_syzygies}
Let $M$ be a graded $\mathbb{L}$-module generated in non-negative degrees not greater than $d\ge 0$.

Then TFAE:
\begin{enumerate}
\item $M$ satisfies the Syzygies Conjecture;

\item For every $j\ge 0$ we have $\Tor^{\LL}_j(M,\LL/\LL^{<0})$ is concentrated 
in degrees $k\colon d\ge k\ge j$
and it is a free abelian group in degree $j$.
\end{enumerate}

In particular, the class of $\mathbb{L}$-modules satisfying 
the Syzygies Conjecture is closed under extensions.
\end{Prop}
\begin{proof}
$(1)\Rightarrow (2)$.
Suppose that $M$ has a free $\LL$-resolution $F_d\rarr F_{d-1}\rarr \cdots \rarr F_0\rarr M$
where $F_j$ is concentrated in degrees between $j$ and $d$.
Tensoring the resolution by $\LL/\LL^{<0}$ over the Lazard ring
we obtain the complex of graded abelian groups $G_d\rarr G_{d-1}\rarr \cdots \rarr G_0$
where $G_j$ is free and is concentrated in degrees between $j$ and $d$.
It is clear now that the $j$-th cohomology of this complex $\Tor_j^\LL(M,\LL/\LL^{<0})$
 can be non-zero only in degrees between $j$ and $d$,
and, moreover, $\Tor_j^\LL(M,\LL/\LL^{<0})^j$ is the kernel of the map $G_j^j\rarr G_{j-1}^j$
and therefore is free.

$(2)\Rightarrow (1)$. The following result is straight-forward.
\begin{Lm}\label{lm:generators_graded_L_module}
Let $N$ be a graded $\mathbb{L}$-module concentrated in degrees bounded from above.

Then if $p:G\rarr N\ot_\LL \LL/\LL^{<0}$ is a surjection of graded abelian groups where $G$ is free,
then there exist a surjection of graded $\LL$-modules
$p_{\LL}:G \ot_\ZZ \LL\rarr N$ s.t. $p_{\LL}\ot_\LL \LL/\LL^{<0}=p$.
\end{Lm}

Suppose the module $M$ satisfying the assumptions (2) of Proposition is given,
the following lemma allows to construct the free $\LL$-resolution of $M$ inductively step by step.
Indeed, in the induction step one assumes that 
$r-1$ steps of the free $\LL$-resolution of $M$
satisfying Syzygies Conjecture are constructed $F_{r-1}\rarr F_{r-2} \rarr \cdots F_0\rarr M$,
and the kernel $N$ of the map $F_{r-1} \rarr F_{r-2}$
satisfies the properties of the following Lemma.
The conclusion of the Lemma together with the assumptions (2)
allows then to continue induction.

\begin{Lm}\label{lm:induction_step}
Let $N$ be a graded $\mathbb{L}$-module, let $r: d\ge r\ge 0$.
Assume that the graded abelian group $N\ot_\LL \LL/\LL^{<0}$
is concentrated in degrees $k\colon d\ge k\ge r$ and
its $r$-th component is free, 
and that $\Tor_1^\LL(N, \LL/\LL^{<0})$ is concentrated in degrees $k\colon d\ge k\ge r+1$
and its $r+1$-th component is free.

Then $N$ has a presentation of $\LL$-modules $0\rarr R\rarr F \rarr N\rarr 0$
where $F$ is free and is generated in degrees $k\colon d\ge k\ge r$,
 $R$ is generated in degrees $k$: $d\ge k\ge r+1$,
 and $(R\ot_\LL \LL/\LL^{<0})^{r+1}$ is free.
 
Moreover, $\Tor_j^{\LL}(R, \LL/\LL^{<0})=\Tor_{j+1}^\LL(N,\LL/\LL^{<0})$ for $j\ge 1$.
\end{Lm}
\begin{proof}
By the assumption one finds a graded free abelian group $G$ 
 concentrated in degrees $k$: $d\ge k\ge r$,
and a map $p:G\rarr N\ot_\LL \LL/\LL^{<0}$ which is an isomorphism in degree $r$.
By Lemma \ref{lm:generators_graded_L_module} 
we have a surjective map $p_\LL:G\ot_\ZZ \LL \rarr N$ of graded $\LL$-modules. 

Denote by $F:=G\ot_\ZZ \LL$ the free module, 
and by $R$ the kernel of the map $p_\LL$. Applying $\ot_\LL \LL/\LL^{<0}$ to 
the exact sequence $0\rarr R\rarr F \rarr N\rarr 0$
we get that $R\ot_\LL \LL/\LL^{<0}$ is zero in degrees less than $r+1$,
and that $\Tor_j^{\LL}(R, \LL/\LL^{<0})=\Tor_{j+1}^\LL(N,\LL/\LL^{<0})$ for $j\ge 1$.

Also we get the exact sequence:
\begin{equation}\label{eq:induction_step}
0\rarr \Tor^{\LL}_{1}(N,\LL/\LL^{<0})^{r+1}\rarr (R\ot_\LL \LL/\LL^{<0})^{r+1}
\rarr (F\ot_\LL \LL/\LL^{<0})^{r+1}, 
\end{equation}
from which it follows that $(R\ot_\LL \LL/\LL^{<0})^{r+1}$
as it is an extension of a free group (its image in a free group $(F\ot_\LL \LL/\LL^{<0})^{r+1}$)
by a free group.
\end{proof}

This proves the equivalence $(1) \Leftrightarrow (2)$.
Suppose now that we have a short exact sequence of $\LL$-modules:
$$ 0\rarr N\rarr M \rarr K\rarr 0,$$
where $N$ and $K$ satisfy  the Syzygies conjecture for the same $d\ge 0$.
We want to show that $M$ also satisfies the Syzygies conjecture for $d$.

Applying $\ot_\LL \LL/\LL^{<0}$ to the short exact sequence
we easily see that $\Tor^{\LL}_j(M,\LL/\LL^{<0})$ is concentrated in degrees $k \colon d\ge k\ge j$.
Moreover,
as we have an exact sequence 
$$\Tor^{\LL}_{j+1}(K,\LL/\LL^{<0})^j=0\rarr \Tor^{\LL}_j(N,\LL/\LL^{<0})^j\rarr \Tor^{\LL}_j(M,\LL/\LL^{<0})^j\rarr \Tor^{\LL}_j(K,\LL/\LL^{<0})^j,$$
where the graded abelian groups on its left and right end are free,
it follows that the group in the middle is also free.
\end{proof}

First, let us check that the Syzygies Conjecture
is true for the simplest $\LL$-modules appearing in Th. \ref{th:filtration}.

\begin{Lm}\label{lm:koszul_resolution}
Let $p$ be a prime, denote by $M$ the graded $\LL$-module $\LL/I(p,n)x$
where $\deg x\ge f(n)$.

Then $\Tor^{\LL}_j(M, \LL/\LL^{<0})$ is concentrated in degrees $k\colon \deg x\ge k\ge j+1$.
\end{Lm}
\begin{proof}
Ideal $I(p,n)$ is regular, and the $\LL$-module $\LL/I(p,n)x$
has a free Koszul resolution $K^\bullet$, s.t. in the $j$-th term of this resolution
we have a free module $\bigwedge^{j} \left(\oplus_{i=0}^{n-1} \LL\cdot v_i\cdot \right) x$. 
Thus, the minimal degree of this free module 
is $\deg x + \sum_{k=n-j}^{n-1}(1-p^k)=\deg x + j - (p^{n-j}+\ldots + p^{n-1})$.
The condition that this degree is at least $j$
is the same as $\deg x\ge p^{n-j}+\ldots + p^{n-1}$.

The strongest case of this inequality is for $j=n$,
 when this is precisely $\deg x\ge f(n)$. 
 This proves that $\Tor^{\LL}_j(M, \LL/\LL^{<0})$ is concentrated in degrees greater than $j$
 for $j\neq n$, and in degrees not less than $n$ for $j=n$.

The map $K^n\ot_\LL \LL/\LL^{<0}\rarr K^{n-1}\ot_\LL \LL/\LL^{<0}$
looks like: 
$$ \ZZ v_0\wedge \cdots \wedge v_{n-1} 
 \rarr \oplus_{i=0}^{n-1} \ZZ v_0\wedge \cdots \hat{v_i}\cdot  \wedge v_{n-1}, $$ 
which sends the generator 
$v_0\wedge \cdots \wedge v_{n-1}$ to $p v_1\wedge \cdots \wedge v_{n-1}$.
Thus, the group $\Tor^{\LL}_n(M, \LL/\LL^{<0})$ is zero.
\end{proof}

\begin{Th}\label{th:syzygies}
Syzygies Conjecture is true, 
and moreover $\Tor^\LL_j(\Omega^*(X), \LL/\LL^{<0})$ equals to zero in degree $j$. 
\end{Th}
\begin{proof}

It is enough to prove that $\Omega^*_{(r)}(X)$ satisfies conditions of Prop. \ref{prop:tor_syzygies}
together with the following: $$\Tor^\LL_j(\Omega^*(X), \LL/\LL^{<0})^j=0.$$
However, by Th. \ref{th:filtration} it is a union of filtered $\LL$-modules 
with graded factors isomorphic to $\LL/I(p,n)x$ as in Lemma \ref{lm:koszul_resolution}.
For these modules $\Tor$'s are zero,
and as $\Tor$ commutes with unions (generally with filtered colimits) we are done.
\end{proof}

In one of few non-trivial examples where $\Omega^*(X)$ can be computed,
namely, for a Pfister quadric, Vishik checked in \cite[Example 4.6]{Vish_Lazard} 
that $\Tor^\LL_j(\Omega^*(X), \LL/(2,\LL^{<0}))$
can be non-zero in degree $j$. It shows the sharpness 
of the obtained estimates on degrees of a free $\LL$-resolution 
of $\Omega^*(X)$.


\medskip

\medskip

\noindent
\sc{Pavel Sechin\\
Mathematisches Insitut, Ruprecht-Karls-Universität Heidelberg, \\
Mathematikon, Im Neuenheimer Feld 205, 69120 Heidelberg, Germany\\
{\tt psechin@mathi.uni-heidelberg.de}
}

\end{document}